\documentclass[10pt, a4paper, reqno,
]{amsart}
\usepackage[margin=2.5cm]{geometry}
\usepackage[foot]{amsaddr}

\makeatletter
\renewcommand{\email}[2][]{%
	\ifx\emails\@empty\relax\else{\g@addto@macro\emails{,\space}}\fi%
	\@ifnotempty{#1}{\g@addto@macro\emails{\textrm{(#1)}\space}}%
	\g@addto@macro\emails{#2}%
}
\makeatother
\usepackage{graphicx} 
\usepackage{amsthm, amsmath, amsfonts, amssymb, amscd, enumerate, enumitem, esint, xcolor, url}
\usepackage[colorlinks]{hyperref}
\usepackage[utf8]{inputenc}
\usepackage{mathtools}

\newtheorem{thm}{Theorem}[section]

\newtheorem{lem}[thm]{Lemma}
\newtheorem{prop}[thm]{Proposition}
\theoremstyle{definition}

\theoremstyle{remark}

\DeclareMathOperator{\supp}{supp}

\newcommand{\loc}{\textup{loc}}
\newcommand{\glob}{\textup{glob}}

\newcommand{\HL}{\textup{HL}}

\newcommand{\RR}{\mathbb{R}}
\newcommand{\NN}{\mathbb{N}}
\newcommand{\CC}{\mathbb{C}}

\newcommand{\KK}{\mathbb{K}}
\newcommand{\TT}{\mathbb{T}}
\newcommand{\GG}{\mathbb{G}}

\newcommand{\blue}[1]{{\color{blue!80} #1}}

\numberwithin{equation}{section}
\usepackage[defaultlines=2, all]{nowidow}
\setnoclub[2]
\allowdisplaybreaks

\setlist[enumerate,1]{label=(\alph*)}

\begin{document}

	\title[Variational inequalities for fractional Kolmogorov operators]
	{Variational inequalities associated with the semigroups generated by fractional Kolmogorov operators}

	\author[J. J. Betancor]{Jorge J. Betancor$^1$}
	\address{$^1$Departamento de An\'alisis Matem\'atico, Universidad de La Laguna,\newline
		Campus de Anchieta, Avda. Astrof\'isico S\'anchez, s/n,\newline
		38721 La Laguna (Sta. Cruz de Tenerife), Spain}
	\email{jbetanco@ull.es}
	
	\author[E. Dalmasso]{Estefan\'ia Dalmasso$^{2}$}
	\address{$^2$Instituto de Matem\'atica Aplicada del Litoral, UNL, CONICET, FIQ.\newline Colectora Ruta Nac. Nº 168, Paraje El Pozo,\newline S3007ABA, Santa Fe, Argentina}
	\email
	{edalmasso@santafe-conicet.gov.ar}
	
	\author[P. Quijano]{Pablo Quijano$^{2}$}
	\email{pquijano@santafe-conicet.gov.ar}
	
	\date{\today}
	\subjclass{42B20, 42B25, 42B35, 47D03}
	
	\keywords{Kolmogorov operator, fractional Laplacian, $\rho$-variation, oscillation and jump operators}

	\begin{abstract}
		In this paper we consider fractional Kolmogorov operators defined, in $\RR^d$, by
		\[\Lambda_\kappa=(-\Delta)^{\alpha/2}+\frac{\kappa}{|x|^\alpha} x\cdot \nabla,\]
		with $\alpha\in (1,2)$, $\alpha<(d+2)/2$ and $\kappa\in \RR$. The operator $\Lambda_\alpha$ generates a holomorphic semigroup $\{T_t^\alpha\}_{t>0}$ in $L^2(\RR^d)$ provided that $\kappa<\kappa_c$ where $\kappa_c$ is a critical coupling constant. We establish $L^p$-boundedness properties for the variation operators $V_\rho\left(\{t^\ell\partial_t^\ell T_t^\alpha\}_{t>0}\right)$ with $\rho> 2$, $\ell\in \NN$ and $1\vee \frac{d}{\beta}<p<\infty$, where $\beta$ depends on $\kappa$. We also study the behavior of these variation operators in the endpoint $L^{1\vee \frac{d}{\beta}}(\RR^d)$ and we prove that $V_2(\{T_t^\alpha\}_{t>0})$ is not bounded from $L^p(\RR^d)$ to $L^{p,\infty}(\RR^d)$ for any $1< p<\infty$. 
	\end{abstract}
	\maketitle

	\section{Introduction}\label{sec: intro}
	
	We consider the fractional Kolmogorov operator formally given by
    \begin{equation*}
        \Lambda_\kappa^\alpha= (-\Delta)^{\alpha/2} + \frac{\kappa}{|x|^\alpha} x\cdot\nabla,\quad\text{ in } \RR^d,
    \end{equation*}
    for $d\in\NN$, $\alpha\in(1,2)$ with $\alpha<(d+2)/2$ and $\kappa\in\RR$. According to the terminology used
    in~\cite{KMS24} and~\cite{KS23} the gradient perturbation is called attractive when the coupling constant $\kappa$ is positive and repulsive when it is negative.

    Our objective in this paper is to establish $L^p$ variational inequalities involving derivatives of the heat semigroup of operators $\{T_t^\alpha\}_{t>0}$ generated by $-\Lambda_\kappa^\alpha$.

    The vector field $\kappa|x|^{-d}x$ has a stronger singularity than the ones of the Kato class. Then, the semigroup $\{T_t^\alpha\}_{t>0}$ is not $L^1-L^{\infty}$ ultracontractive, but it is ultracontractive as a mapping from $L^1(\RR^d,\varphi_s dx)$ into $L^\infty(\RR^d,dx)$ where 
    \[\varphi_s(y) = (s^{-1/\alpha}|y|\wedge 1)^{-d+\beta}, \quad 0<t\leq s.\] This fact is crucial in the proof of pointwise estimates for the integral kernel $T_t^\alpha(x,y)$ of $T_t^\alpha$ for $x,y\in\RR^d$ and $t>0$ (see~\cite{KSS21}). In contrast with other studies (see~\cite{ABR23, BDQ, BD23, BM, FMS21, KMV+18, MER21, Na23}, for instance), in this work we deal with nonlocal and nonsymmetric situation.

    We consider the function 
    \begin{equation}\label{eq: def-psi}
        \Psi(\beta) = \frac{2^\alpha}{\beta-\alpha} \frac{\Gamma\left(\frac{\beta}{2}\right)\Gamma\left(\frac{d}{2}-\frac{\beta-d}{2}\right)}{\Gamma\left(\frac{\alpha}{2}-\frac{\beta}{2}\right)\Gamma\left(\frac{\beta-d}{2}\right)}, \quad \beta\in(\alpha,d+\alpha).
    \end{equation}

    The fraction of the four gamma functions in~(\ref{eq: def-psi}) achieves its maximum value in the point $\beta = (\alpha + d)/2$ (see~\cite{FMS21}). The function $\Psi$ is not symmetric around the point $(\alpha + d)/2$ (see~\cite[p. 347]{KS23} and \cite[p. 1868]{KSS21} for plots). It was proved in~\cite[Lemma 1.1]{BDM} that $\Psi$ is strictly decreasing on the interval $((d+\alpha)/2,d+\alpha)$. 
    
    We denote by $\kappa_c = \Psi((d+\alpha)/2)$. The value $\kappa_c$ is critical in the following sense: for $\kappa<\kappa_c$ the Hardy drift is a Rellich perturbation of $(-\Delta)^{\alpha/2}$, i.e. 
    \begin{equation*}
        \left\| \frac{\kappa}{|x|^{-\alpha}} x\cdot \nabla (\xi + (-\Delta)^{\alpha/2})^{-1}\right\|_{L^{2}(\RR^d,dx)\rightarrow L^2(\RR^d,dx)}<1,\quad \kappa<\kappa_c \text{ and } \xi>0.
    \end{equation*}
    Then, $\{T_t^\alpha\}_{t>0}$ is a contractive and holomorphic semigroup in $L^2(\RR^d,dx)$ when $\kappa<\kappa_c$ (see~\cite[\S 7]{KS20}).

    Kinzebulatov, Semenov and Szczypkowski 
    \cite{KS20, KS23, KSS21} proved that $\{T_t^\alpha\}_{t>0}$ is contractive in $L^{\infty}(\RR^d,dx)$ and it extends by continuity to a $C_0$-semigroup in $L^p(\RR^d,dx)$ when $p\in [2,\infty)$, provided $\kappa<\kappa_c$. 

    Suppose that $\{S_t\}_{t>0}$ is a family of operators defined on $L^p(\RR^d,dx)$, for some $1\leq p<\infty$. Let $\rho>0$. The $\rho$-variation operator $V_\rho(\{S_t\}_{t>0})$ associated with $\{S_t\}_{t>0}$ is defined by 
    \begin{equation*}
        V_\rho(\{S_t\}_{t>0})(f)(x)  = \sup_{\substack{ 0<t_k<\dots<t_1<\infty \\ k\in\NN}}
        \left( \sum_{j=1}^{k-1} \left| S_{t_j}(f)(x) - S_{t_{j+1}}(f)(x)\right|^\rho \right)^{1/\rho}.
    \end{equation*}

    We are interested in $L^p$ boundedness properties of $\rho$-variation operators. We have to see first that 
    $V_\rho(\{S_t\}_{t>0})(f)$ defines a Lebesgue measurable function on $\RR^d$. In order to prove this, we will use some $t$-continuity properties of $\{S_t\}_{t>0}$ (see~\cite[p. 60]{CJRW1}). The $\rho$-variation operators are usually considered for $\rho > 2$ because $L^p$-boundedness has not been proven for $\rho = 2$ and, in some cases (see~\cite{CCSj} and~\cite{Qi}), it has even been disproved. 

    When the exponent $\rho = 2$, it is common to replace the variation operator $V_2$ with the oscillation operator, defined as follows. Suppose that $\{t_j\}_{j\in\NN}$ is a decreasing sequence of positive numbers such that $\lim_{j\rightarrow\infty} t_j = 0$. We define $O(\{S_t\}_{t>0},\{t_j\}_{j\in\NN})$ by
    \begin{equation*}
        O(\{S_t\}_{t>0},\{t_j\}_{j\in\NN})(f)(x) = 
        \left(
        \sum_{j\in\NN} \sup_{t_{j+1}\leq s_{j+1}<s_j\leq t_j} |S_{s_j}(f)(x) - S_{s_{j+1}}(f)(x)|^2
        \right)^{1/2}.
    \end{equation*}

    Lepingle~\cite{Lep} established variational inequalities for martingales in $L^2$. Lepingle's inequality was extended to $L^p$, $p>1$, in~\cite{JKRW}. Bourgain \cite{Bo1, Bo2, Bo3} was the first to prove variational inequalities in ergodic theory. After Bourgain's papers, variational inequalities appear in many others, some of them very recent ones, about probability, ergodic theory and harmonic analysis. In particular, variation operators associated with truncations of singular integrals and semigroups of operators have been established (see, for example, \cite{AB, CJRW1, CJRW2, CMMTV, JR, JW, LeMX1, MTX1, MTX2, OSTTW, ZK}). A result about the $\rho$-variation operator always gives information about convergence properties for the family of operators under consideration. Those readers interested in variational inequalities and their applications can find information in~\cite{BORSS, BMSW, 
    PDU, DL, GT, IMMS, JSW, KMT, MT, MS24, MSS, MST, MSZ, MTZ, Na23, PX}.

    We assume throughout this paper that $\kappa<\kappa_c$. As it was mentioned, $\{T_t^\alpha\}_{t>0}$ is a holomorphic and contractive semigroup in $L^2(\RR^d)$. Also, for every $t>0$, the operator $T_t^\alpha$ is positive (in the lattice sense). According to~\cite[Corollary~4.5, and p.~2092]{LeMX1} (see also \cite[p. 2075]{LeMX1}) we deduce the following result.

    \begin{thm}\label{thm: L2 var osc}
        Let $\alpha\in(1,2\wedge (d+2)/2)$, $\ell\in\NN\cup\{0\}$ and ${\rho>2}$. Suppose that $\{t_j\}_{j\in\NN}$ is a decreasing sequence of positive numbers such that $\lim_{j\rightarrow \infty}t_j = 0$. Then, the operators $V_\rho(\{t^\ell\partial^ \ell_t T_t^\alpha\}_{t>0})$ and $ O(\{t^\ell\partial^ \ell_t T_t^\alpha\}_{t>0},\{t_j\}_{j\in\NN})$ are bounded on $L^2(\RR^d)$.
    \end{thm}

    For every $\ell \in \mathbb{N} \cup \{0\}$, we denote by $T_{t,\ell}^\alpha(x,y)$, for $x,y \in \mathbb{R}^d$ and $t > 0$, the kernel of the integral operator $t^\ell \partial_t^\ell T_t^\alpha$. In order to extend the boundedness properties in Theorem~\ref{thm: L2 var osc} to $L^p(\RR^d)$ with $p\neq 2$, the following pointwise estimate for $T_{t,\ell}(x,y)$ established in~\cite[Proposition 2.9]{BDM} is a key result:
    \begin{equation*}
        \left|T_{t,\ell}^\alpha(x,y)\right| \leq \frac{C}{t^{d/\alpha}}
        \left( \frac{t^{1/\alpha}}{t^{1/\alpha}+|x-y|} \right)^{d+ \alpha - \epsilon} \left(
        1+ \frac{t^{1/\alpha}}{|y|} \right)^{d-\beta},
        \quad x,y\in\RR^d \text{ and } t>0.
    \end{equation*}
    The constant $C$ depends on $\ell\in\NN\cup\{0\}$ and $\epsilon>0$.

    \begin{thm}\label{thm: var osc strong-weak}
        Let $\alpha\in(1,2\wedge (d+2)/2)$, $\beta \in ((d+\alpha)/2,d+\alpha)$, ${\kappa = \Psi(\beta)}$, $\ell\in\NN\cup\{0\} $ and $\rho>2$. Suppose that $\{t_j\}_{j\in\NN}$ is a decreasing sequence of positive numbers such that $\lim_{j\rightarrow \infty}t_j = 0$. The operators $V_\rho(\{t^\ell\partial^ \ell_t T_t^\alpha\}_{t>0})$ and $ O(\{t^\ell\partial^ \ell_t T_t^\alpha\}_{t>0},\{t_j\}_{j\in\NN})$ are bounded
        \begin{enumerate}
            \item\label{itm: var osc strong} on $L^p(\RR^d)$ provided that $1\vee \frac{d}{\beta}<p<\infty$,
            \item  
            from $L^{1}(\RR^d)$ into $L^{1,\infty}(\RR^d)$ when $d\leq \beta$,
        \end{enumerate}
        Furthermore, $V_\rho(\{T_t^\alpha\}_{t>0})$ is not bounded from $L^{d/\beta}(\RR^d)$ into $L^{d/\beta,\infty}(\RR^d)$ when $d > \beta$.
    \end{thm}

 We denote by $\{W_t\}_{t>0}$ the classical heat semigroup associated with the Euclidean Laplacian $-\Delta$. Let $\ell\in \NN\cup\{0\}$, $\rho>2$ and $\{t_j\}_{j\in\NN}$ be a decreasing sequence of positive numbers such that $\lim_{j\rightarrow \infty} t_j = 0$. In~\cite[p.~12]{AB} it was established that the operators $V_{\rho}(\{t^\ell\partial^\ell_t W_t\}_{t>0})$ and $O(\{t^\ell\partial^\ell_t W_t\}_{t>0}, \{t_j\}_{j\in\NN})$ are bounded on $L^p(\RR^d)$ for every $1<p<\infty$, and from $L^1(\RR^d)$ into $L^{1,\infty}(\RR^d)$. 

 Given $\alpha\in (0,2)$, we represent by $\{W_t^\alpha\}_{t>0}$ the heat semigroup associated with the fractional power $(-\Delta)^{\alpha/2}$ defined by the subordination formula in the following way
 \begin{equation*}
     W_t^\alpha(f) = \int_0^\infty \eta_{\alpha/2,t}(s)W_s(f)ds,
 \end{equation*}
where, for every $\theta\in (0,1)$,
\begin{equation*}
    \eta_{\theta,t}(s) = \frac{1}{2\pi i} \int_{\sigma-i\infty}^{\sigma+i\infty} e^{zs-tz^{\theta}}dz, \quad t,s>0,
\end{equation*}
and the last integral does not depend on the choice of $\sigma>0$ (see~\cite[p. 259]{Yos}). By using Corollary~4.5 and Corollary~6.2 in~\cite{LeMX1} (see also~\cite[p. 3]{BDQ}) we can see that the operators $V_{\rho}(\{t^\ell\partial^\ell_t W_t^\alpha\}_{t>0})$ and $O(\{t^\ell\partial^\ell_t W_t^\alpha\}_{t>0}, \{t_j\}_{j\in\NN})$ are bounded on $L^p(\RR^d)$ for every $1<p<\infty$. We will prove in Proposition~\ref{prop: weak type 1-1} (Section~\ref{sec: weak type var osc}) that these operators are also bounded from $L^1(\RR^d)$ into $L^{1,\infty}(\RR^d)$.

In order to prove the positive results of Theorem~\ref{thm: var osc strong-weak} for $\rho$-variation operators, we will show that the operators $V_{\rho}(\{t^\ell \partial_t^\ell (T_t^\alpha - W_t^\alpha)\}_{t>0})$, for $\rho > 2$ (actually for $\rho\ge 1$) are bounded in $L^p(\RR^d)$, when $1\vee \frac{d}{\beta}<p<\infty$ (Proposition \ref{prop: difference} in Section~\ref{sec: difference}  and from $L^1(\mathbb{R}^d)$ into $L^{1,\infty}(\mathbb{R}^d)$ in Section~\ref{sec: endpoint var}, Proposition \ref{prop: weak difference}). The negative result in Theorem \ref{thm: var osc strong-weak} is also proved in Section~\ref{sec: endpoint var}. The properties in Theorem \ref{thm: var osc strong-weak} for the oscillation operators can be proved as those for $\rho$-variation operators (see Section~\ref{sec: endpoint osc}).

Suppose that $\{S_t\}_{t>0}$ is a family of operators defined in $L^p(\RR^d)$ for some $1\le p<\infty$ and $\lambda>0$. The $\lambda$-jump operator $J(\{S_t\}_{t>0},\lambda)$ is given by
$$
J(\{S_t\}_{t>0},\lambda)(f)(x)= \sup \,\mathcal{G}(\{S_t\}_{t>0},\lambda)(f)(x),
$$
where, for every $x\in \RR^d$, $\mathcal{G}(\{S_t\}_{t>0},\lambda)(f)(x)$ consists of all those $n\in \NN$ such that there exist $0<s_1<t_1\le s_2<t_2\le \ldots \le s_n<t_n$ satisfying that $|S_{t_i}(f)(x)-S_{s_i}(f)(x)|>\lambda$, $i=1,\dots, n$.

In the following result we establish $L^p$-boundedness properties for $J(\{t^\ell\partial_t^\ell T_t^\alpha\}_{t>0},\lambda)$, $\ell\in \NN\cup \{0\}$, which will be proved in Section~\ref{sec: jump}.

\begin{thm}\label{thm: jump strong-weak}
        Let $\alpha\in(1,2\wedge (d+2)/2)$, $\beta \in ((d+\alpha)/2,d+\alpha)$, ${\kappa = \Psi(\beta)}$, $\ell\in\NN\cup\{0\}$ and $\rho> 2$. The operator $\lambda[J(\{t^\ell \partial_t^\ell T_t^\alpha\}_{t>0},\lambda)]^{1/\rho}$ and $\lambda [J(\{T_t^\alpha\}_{t>0},\lambda)]^{1/2}$ are bounded in $L^p(\RR^d)$, uniformly in $\lambda>0$ for every $1\vee \frac{d}{\beta}<p<\infty$. Furthermore, $\lambda [J(\{t^\ell \partial_t^\ell T_t^\alpha\}_{t>0},\lambda)]^{1/\rho}$ is bounded from $L^1(\RR^d)$ into $L^{1,\infty}(\RR^d)$, uniformly in $\lambda>0$, when $d\le \beta$.
\end{thm}

 Our next objective is to prove weighted $L^p$-inequalities for $\rho$-variation and oscillation operators involving $\{t^\ell \partial^\ell_t T_t^\alpha\}_{t>0}$, $\ell\in\NN\cup\{0\}$. Before stating our results we recall some definitions. A weight in $\RR^d$ is a nonnegative measurable function on $\RR^d$. If $1<p<\infty$ we say that a weight $w$ is in the Muckenhoupt class $A_p(\RR^d)$ when
 \begin{equation*}
     \sup_B \left( \frac{1}{|B|} \int_B w\right) 
     \left( \frac{1}{|B|} \int_B w^{-1/(p-1)} \right)^{p-1} <\infty,
 \end{equation*}
 where the supremum is taken over all balls $B\subset\RR^d$.

Hereafter we denote by $L^p(\RR^d,w)$ the $p$-Lebesgue space with weight $w$. 

\begin{thm}\label{thm: weighted Lp}
Let $\alpha\in (1,2\wedge (d+2)/2)$, $\beta\in ((d+\alpha)/2,d+\alpha)$, $\kappa=\Psi(\beta)$, $\ell\in \NN\cup\{0\}$ and $\rho>2$. Suppose that $\{t_j\}_{j\in\NN}$ is a decreasing sequence of positive numbers such that $\lim_{j\rightarrow \infty}t_j = 0$. The operators $V_\rho(\{t^\ell\partial^ \ell_t T_t^\alpha\}_{t>0})$ and $ O(\{t^\ell\partial^ \ell_t T_t^\alpha\}_{t>0},\{t_j\}_{j\in\NN})$ are bounded on $L^p(\RR^d,w)$ for $1\vee d/\beta <p<\infty$ and any weight $w\in A_{p/(1\vee d/\beta)}(\mathbb{R}^d)$. Furthermore, the operator $\lambda \left[J(\{t^\ell\partial _t^\ell T_t^\alpha\}_{t>0},\lambda)\right]^{1/\rho}$ is bounded in $L^p(\RR^d,w)$, uniformly in $\lambda>0$, for $1\vee d/\beta <p<\infty$ and any weight $w\in A_{p/(1\vee d/\beta)}(\mathbb{R}^d)$.
\end{thm}

This theorem will be proved in Section~\ref{sec: weighted} by using \cite[Theorem~6.6]{BZ}.

As it was mentioned, the $2$-variation operator may behave differently from the $\rho$-variation operators in $L^p$ spaces when $\rho>2$ (see~\cite{CCSj} and~\cite{Qi}). The condition $\rho>2$ is necessary for the $\rho$-variation operators $V_\rho(\{T_t^\alpha\}_{t>0})$ and $V_\rho(\{W_t^\alpha\}_{t>0})$ to be bounded from $L^p(\mathbb{R}^d)$ into $L^{p,\infty}(\mathbb{R}^d)$ for $1<p<\infty$.

\begin{thm}\label{thm: necessity rho>2}
 The operators
    \begin{enumerate}
        \item $V_2(\{W_t^\alpha\}_{t>0})$ with $0<\alpha<2$ and $1<p<\infty$;
    \item $V_2(\{T_t^\alpha\}_{t>0})$ with $1<\alpha\wedge (d+2)/2$ and $1\vee d/\beta<p<\infty$,
    \end{enumerate}
    are not bounded from $L^p(\RR^d)$ into $L^{p,\infty}(\RR^d)$.
\end{thm}

The semigroup $\{W_t^\alpha\}_{t>0}$ is defined by subordination. This fact requires a more demanding argument than the one used in the proof of \cite[Theorem 8.1]{CCSj} in order to establish Theorem~\ref{thm: necessity rho>2} for $V_2(\{W_t^\alpha\}_{t>0})$. To prove the corresponding property in Theorem \ref{thm: necessity rho>2} for $V_2(\{T_t^\alpha\}_{t>0})$, we rely on the boundedness of the operator $V_2(\{T_t^\alpha-W_t^\alpha\}_{t>0})$ on $L^p(\mathbb{R}^d)$, provided that $1\vee \frac{d}{\beta}<p<\infty$ (see Proposition~\ref{prop: difference}).

Throughout this paper by $C$ or $c$ we will always denote positive constants that can change from one line to the other. We write $a \lesssim b$ to indicate that there exists a positive constant $c$ such that $a \leq c b$. Similarly, $a \gtrsim b$ means $b \lesssim a$. When both relations hold, we simply write $a \sim b$.

   \section{Weak type estimates for the variation and oscillation operators related to the fractional heat semigroup}\label{sec: weak type var osc}

       As we announced, in this section we will prove the following result.
    \begin{prop}\label{prop: weak type 1-1}
        Let $\alpha\in (0,2)$ and $\ell\in \NN\cup\{0\}$. Then, the operators $V_{\rho}(\{t^\ell\partial^\ell_t W_t^\alpha\}_{t>0})$, for ${\rho>2}$, and $O(\{t^\ell\partial^\ell_t W_t^\alpha\}_{t>0}, \{t_j\}_{j\in\NN})$, where  $\{t_j\}_{j\in\NN}$ is a decreasing sequence of real numbers such that ${\lim_{j\rightarrow\infty}t_j = 0}$, are both bounded from $L^1(\RR^d)$ into $L^{1,\infty}(\RR^d)$. 
    \end{prop}

    In order to prove this proposition, we use the vector-valued Calder\'on--Zygmund theory for singular integrals. If $1 \leq p \leq \infty$ and $E$ is a Banach space, we denote by $L^p(\mathbb{R}^{d}, E)$ the $p$-Bochner--Lebesgue space. We first introduce some notation.

    Suppose that $g:(0,\infty)\rightarrow \CC$. We define, for $\rho>0$,
    \begin{equation*}
        \|g\|_\rho = \sup_{\substack{0<t_1<\dots<t_n \\ n\in\NN }} \left(  \sum_{j=1}^{n-1} |g(t_{j+1}) - g(t_j)|^\rho
        \right)^{1/\rho}.
    \end{equation*}
    It is clear that $\|g\|_\rho = 0$ if and only if $g$ is a constant function. We define the space $E_{\rho}$ that consists of those $g:(0,\infty)\rightarrow \CC$ such that $\|g\|_\rho<\infty$. By identifying each pair of functions that differ by a constant, $(E_\rho,\|\cdot\|_{\rho})$ is a Banach space.

    If $g:(0,\infty)\rightarrow \CC$ is derivable, for every$\rho\geq 1$ it follows that 
    \begin{equation}\label{eq: ineq var-deriv}
\|g\|_\rho\le \int_0^\infty |g'(t)|dt.
    \end{equation}
    
    \begin{proof}[Proof of Proposition~\ref{prop: weak type 1-1}]

    We write
    \begin{equation*}
        W_t^\alpha(f)(x) = \int_{\RR^ d} W_t^\alpha(x-y)f(y)dy, \quad x\in\RR^d \text{ and } t>0,
    \end{equation*}
where 
\[W_t^\alpha(z)=\int_0^\infty \eta_{\alpha/2,t}(s)W_s(z)ds,\quad z\in\RR^d \text{ and } t>0.\]

Let $f\in L^1(\RR^d)$ and $\ell\in \mathbb{N}$. According to \cite[Proposition~2.2]{BDQ} we obtain, for every $0<\epsilon<1$,
\begin{equation*}
\int_{\RR^d} |\partial_t^\ell W_t^\alpha(x-y)||f(y)|dy \leq \frac{C}{t^{\ell+d/\alpha}}\int_{\RR^d}\left(
        \frac{t^{1/\alpha}}{t^{1/\alpha}+|x-y|}
        \right)^{d+\alpha-\epsilon}
         |f(y)| dy 
         \leq C t^{-(\ell+d/\alpha)}\|f\|_{L^1(\RR^d)}
\end{equation*}
for $x\in \RR^d$ and $t>0$. Then, we get  
    \begin{equation*}
        \partial^\ell_t W_t^\alpha(f)(x) = \int_{\RR^d} \partial^\ell_t W_t^\alpha(x-y) f(y) dy, \quad x\in\RR^d \text{ and } t>0.
    \end{equation*}
    From (\ref{eq: ineq var-deriv}) we have that, for every $z\in\RR^d\setminus\{0\}$, 
    \begin{equation*}
        \| t^\ell \partial^\ell_t W_t^\alpha(z)\|_\rho 
        \leq \int_0^\infty |\partial_t(t^\ell\partial^\ell_t W_t^\alpha(z))|dt.
    \end{equation*}
    By using again \cite[Proposition~2.2]{BDQ} it follows that, if $0<\epsilon<1$ and $k\in \NN$,
    \begin{align*}
        \int_0^\infty t^{k-1} |\partial^{k}_t W_t^\alpha(z)|dt & \lesssim \int_0^\infty t^{-d/\alpha - 1} \left(\frac{t^{1/\alpha}}{t^{1/\alpha}+|z|}\right)^{d+\alpha-\epsilon}dt \lesssim \frac{1}{|z|^d}.
    \end{align*}
	for $z\in\RR^d\setminus\{0\}$. Then, 
    \begin{equation}\label{eq: size derivatives W}
        \|t^\ell\partial^\ell_t W_t^\alpha(z) \|_\rho \leq\frac{C_\ell}{|z|^d}, \quad 
        z\in\RR^d\setminus\{0\}.
    \end{equation}

    By proceeding in a similar way to~\cite[Lemma~2.11]{BDM}, we obtain that 
    \begin{equation}\label{eq: size gradient W}
        |\nabla_1 \partial^\ell_t W_t^\alpha (z)|\leq t^{-\frac{d}{\alpha}-\ell-1} \left(\frac{t^{1/\alpha}}{t^{1/\alpha}+|z|}\right)^{d+\alpha+1-\epsilon}, \quad z\in\RR^d\setminus\{0\},
    \end{equation}
therefore
\begin{equation*}
   \|\nabla_1 \partial^\ell_t W_t^\alpha (z)\|_\rho \leq 
        \frac{C}{|z|^{d+1}},\quad z\in\RR^d\setminus\{0\}.
\end{equation*}
    
    We define now the operator
    \begin{equation*}
        \mathcal{L}_\ell(f)(x) = \int_{\RR^d} H_{\ell}(x-y)f(y)dy,\quad x\in\RR^d,
    \end{equation*}
    where for every $z\in\RR^d\setminus\{0\}$,
        $H_{\ell}(z):(0,\infty)\rightarrow \CC$ is given by $[H_{\ell}(z)](t) = t^{\ell}\partial^\ell_t W_t^\alpha(z)$,
  and the integral defining $\mathcal{L}_\ell$ is understood in an $E_\rho$-Bochner sense. 

  Suppose that $f\in L^{\infty}(\RR^d)$ with compact support. According to~\eqref{eq: size derivatives W} we have that
  \begin{equation*}
      \int_{\RR^d} \|H_{\ell}(x-y)\|_{\rho} |f(y)| dy <\infty, \quad x\notin \textup{supp}(f).
  \end{equation*}

  We are going to see that, for $x\notin \textup{supp}(f)$,
  \begin{equation*}
      [\mathcal{L}_\ell(f)(x)](t) = t^k\partial^\ell_t W_t^\alpha(f)(x), \quad \textup{a.e. } t\in(0,\infty).
  \end{equation*}
  Let $a$, $b\in(0,\infty)$, $a\neq b$. We define
  \begin{equation*}
      M[g] = g(b)-g(a),\quad g\in E_\rho.
  \end{equation*}
    Thus, $M\in E'_\rho $, the dual space of $E_\rho$. We have that
    \begin{align*}
            M[\mathcal{L}_\ell(f)(x)] & = \int_{\RR^d} M[H_{\ell}(x-y)]f(y)dy 
            \\ & = \int_{\RR^d} \left[
            \left. t^\ell\partial^\ell_t  W_t^\alpha(x-y)\right|_{t=b} - 
            \left. t^\ell\partial^\ell_t W_t^\alpha(x-y)\right|_{t=a}
            \right] f(y) dy
            \\ & = 
            \left.t^\ell\partial^\ell_t W_t^\alpha(f)(x)\right|_{t=b} 
            - \left.t^\ell\partial^\ell_t W_t^\alpha(f)(x)\right|_{t=a}.
        \end{align*}
    It follows that, for fixed $a\in (0,\infty)$, 
    \begin{equation*}
        [\mathcal{L}_\ell(f)(x)](t) - t^k\partial^\ell_t W_t^\alpha (f)(x) = [\mathcal{L}_\ell(f)(x)](a) -  \left.t^\ell\partial^\ell_t W_t^\alpha (f)(x)\right|_{t=a}, \quad t>0.
    \end{equation*}
    Then, $\mathcal{L}_\ell(f)(x) = \mathbb{L}_\ell(f)(x)$ in $E_\rho$, where $[\mathbb{L}_\ell(f)(x)](t) =  t^\ell\partial^\ell_t W_t^\alpha (f)(x)$, $t>0$.

The operator $\mathbb{L}_{\ell}$ is bounded from $L^{p}(\RR^d)$ into $L^{p}\left(\mathbb{R}^{d}, E_{\rho}\right)$, for every $1<p<\infty$. By~\eqref{eq: size derivatives W} and~\eqref{eq: size gradient W}, $t^{\ell} \partial_{t}^{\ell} W_t^\alpha(z)$ is a $E_{\rho}$-Calderón--Zygmund kernel. Then, we conclude that $\mathbb{L}_{\ell}$ is bounded from $L^{1}(\RR^d)$ into $L^{1, \infty}\left(\mathbb{R}^{d}, E_{p}\right)$, that is, the $\rho$-variation operator $V_{\rho}\left(\left\{t^{\ell} \partial_{t}^{\ell} W_{t}^{\alpha}\right\}_{t>0}\right)$ is bounded from $L^{1}(\RR^d)$ into $L^{1, \infty}(\RR^d)$.

By proceeding in a similar way we can establish that the oscillation operator $O\left(\left\{t^{\ell} \partial_{t}^{\ell} W_{t}^{\alpha}\right\}_{t>0}, \{t_{j}\}_{j \in \mathbb{N}}\right)$ is bounded from $L^{1}(\RR^d)$ into $L^{1, \infty}(\RR^d)$.
 \end{proof}

     \section{Proof of Theorem~\ref{thm: var osc strong-weak}\textnormal{\ref{itm: var osc strong}} for \texorpdfstring{$\rho$}{ρ}-variation operators}\label{sec: difference}

In order to prove the $L^{p}$-boundedness properties for the $\rho$-variation operator $V_{\rho}\left(\left\{t^{k} \partial_{t}^{k} T_t^\alpha\right\}_{t>0}\right)$ we consider the $\rho$-variation operators of the difference
\[
V_{\rho}\left(\left\{t^{k} \partial_{t}^{k}\left(W_{t}^{\alpha}-T_t^\alpha)\right)\right\}_{t>0}\right),  \quad k \in \NN, \rho>2.
\]
Note that the next result holds for every $\rho\ge 1$.
\begin{prop}\label{prop: difference} Let $\alpha\in (1,2\wedge (d+2)/2)$, $\beta\in ((d+\alpha)/2,d+\alpha)$, and $\kappa=\Phi(\beta)$, $k\in\NN\cup \{0\}$ and $\rho\ge 1$. The $\rho$-variation operator $V_{\rho}\left(\left\{t^{k} \partial_{t}^{k}\left(W_{t}^{\alpha}-T_t^\alpha)\right)\right\}_{t>0}\right)$ is bounded in $L^p(\RR^d)$, for every $1\vee \frac{d}{\beta}<p<\infty$.
\end{prop}

\begin{proof} By using (\ref{eq: ineq var-deriv}) we have that
\begin{equation*}
    \begin{split}
        V_\rho & (\{t^k\partial^k_t(W_t^\alpha -T_t^\alpha)\}_{t>0})(f)(x) \\ & =  \sup_{\substack{0<t_1<\dots<t_n \\ n\in\NN}} \left( \sum_{j=1}^{n-1} \left|\left.t^k\partial^k_t (W_t^\alpha -T_t^\alpha)\}_{t>0})(f)(x)\right|_{t=t_{j+1}} - \left.t^k\partial^k_t (W_t^\alpha -T_t^\alpha)\}_{t>0})(f)(x)\right|_{t=t_j}\right|^\rho\right)^{1/\rho}
        \\ & \leq   \int_{0}^{\infty} \left|\partial_t\left(t^k\partial^k_t (W_t^\alpha -T_t^\alpha)\}_{t>0})(f)(x)\right)\right| dt.
    \end{split}
\end{equation*}
Note that the last estimate holds for every $\rho \geq 1$ (not only for $\rho>2$).

We define, for every $\ell \in \mathbb{N} \cup\{0\}$ and $f\in L^p(\mathbb{R})$, with $\frac{d}
{\beta}\vee 1<p<\infty$, the operator 
\[
\TT _{\ell}(f)(x)=\int_{\mathbb{R}^{\alpha}} f(y) \int_{0}^{\infty} \left| t^{\ell} \partial_{t}^{\ell+1}\left(W_{t}^{\alpha}(x - y)-T_t^\alpha(x, y)\right) \right| dt dy, \quad x \in \mathbb{R}^d.
\]

Let $k \in \mathbb{N}$ and $f \in L^{p}(\RR^d)$ with $1<p<\infty$. According to \cite[Proposition~2.8]{BDM} we obtain, for every $\epsilon>0$,
\begin{align*}
\int_{\RR^d} |\partial_t^kW_t^\alpha(x-y)||f(y)|dy
&  \lesssim \frac{1}{t^{k+d/\alpha}}\int_{\RR^d}\left(
        \frac{t^{1/\alpha}}{t^{1/\alpha}+|x-y|}
        \right)^{d+\alpha-\epsilon}
        \left(1+\frac{t^{1/\alpha}}{|y|}\right)^{d-\beta} |f(y)| dy 
        \\ & \lesssim \|f\|_{L^p(\RR^d)}
        \left(
        \int_{\RR^d}\left(
        \frac{t^{1/\alpha}}{t^{1/\alpha}+|x-y|}
        \right)^{(d+\alpha-\epsilon)p'}
        \left(1+\frac{t^{1/\alpha}}{|y|}\right)^{(d-\beta)p'}
        dy\right)^{1/p'}
    \end{align*}
for $x\in \RR^d$ and $t>0$. Notice that
\[-(d-\beta)p'+d-1>-1\,\Longleftrightarrow\, (d-\beta)p<d(p-1) \,\Longleftrightarrow\, \beta p-d>0.\]
Also, we have
$(d+\alpha-\epsilon)p'-d>0$ when $0<\epsilon<\alpha$. Hence, if $\epsilon>0$ is small enough we can take derivatives with respect to $t$ under the integral sign when $p>1\vee \frac{d}{\beta}$, to get
\begin{equation*}
    \partial_t^k \int_{\RR^d} W_t^\alpha(x-y)f(y)dy = \int_{\RR^d}\partial^k_t W_t^\alpha(x-y) f(y) dy, \quad x\in \RR^d.
\end{equation*}
Then, Fubini's theorem leads to
\begin{equation*}
    V_\rho(\{t^k\partial^k_t(W_t^\alpha -T_t^\alpha)\}_{t>0})(f)(x) \leq k \TT _{k-1}(|f|)(x)+ \TT_k (|f|)(x),\quad x\in \RR^d.
\end{equation*}
Note that the first term in the right-hand side does not appear when $k=0$. 

Let $\ell\in \NN\cup\{0\}$. Our next objective is to prove the $L^p$-boundedness properties for the operator $\TT _\ell$. We decompose $\RR^d \times \RR^d \times (0,\infty)$ into two regions that will be named the \textit{local part} and the \textit{global part}. The local part $\Omega_{\loc}$ of $\RR^d \times\RR^d \times (0,\infty)$ is defined by
\begin{equation*}
    \Omega_{\loc} = \left\{ (x,y,t)\in \RR^d \times\RR^d \times (0,\infty): t\leq |y|^\alpha \text{ and } |x-y|\leq \frac{|x|\wedge |y|}{2}\right\}
\end{equation*}
 and the global part $\Omega_{\textup{glob}}$ is defined by $\Omega_{\textup{glob}} = (\RR^d \times\RR^d \times (0,\infty)) \setminus \Omega_{\loc}$.

 Duhamel's formula allows us to write 
 \begin{equation*}
     \begin{split}
         T_t^\alpha(x,y) - W_t^\alpha(x-y) &  = -k
         \int_0^t \int_{\RR^d} T_{t-s}^{\alpha}(x,z) |z|^{-d} z\cdot \nabla_1 W_t^\alpha(z-y)dzds
          \\ & = -k \int_0^{t/2} \int_{\RR^d} T_{t-s}^{\alpha}(x,z) |z|^{-d} z\cdot \nabla_1 W_t^\alpha(z-y)dzds
          \\ &\quad -k \int_0^{t/2} \int_{\RR^d} T_t^\alpha(x,z) |z|^{-d} z\cdot  \nabla_1 W^\alpha_{t-s}(z-y)dzds,
     \end{split}
 \end{equation*}
for $x, y\in\RR^d$ and $t>0$.

Then,
\begin{equation*}
    \begin{split}
        |\partial^k_t (T_t^\alpha(x,y) - W_t^\alpha(x-y))|& \leq C \left( \sum_{j=0}^{k-1} \int_{\RR^d}|\partial_t^j T_t^\alpha(x,z)| |z|^{-d} |z\cdot\nabla_1\partial^{k-1-j}_t W^\alpha_{t/2}(z-y)|dz\right.
        \\ & \quad + 
        \int_0^{t/2} \int_{\RR^d} |\partial^k_u T_u^\alpha(x,z)|_{u=t-s} |z|^{-d}|z\cdot\nabla_1 W^\alpha_s(z-y)| dzds 
         \\ &\quad  \left. + 
        \int_0^{t/2} \int_{\RR^d} |T_s^{\alpha}(x,z)| |z|^{-d}|z\cdot\nabla_1 \partial^k_u W^\alpha_u(z-y)|_{u=t-2} dzds\right)
        \\ & := \sum_{j=0}^{k+1} K_j(t,x,y), \quad x,y\in\RR^d \text{ and } t>0,
    \end{split}
\end{equation*}
where, for every $x, y \in \mathbb{R}^{d}$ and $t>0$, 
\begin{align*}
K_{j}(t, x, y)&=\int_{\mathbb{R}^{d}}\left|\partial_{t}^{j} T_{t / 2}^{\alpha}(x, z)\right||z|^{-d}\left|z \cdot \nabla_1 \partial_{t}^{k-1-j} W_{t / 2}^{\alpha}(y- z)\right| dz, \quad j=0, \dots, k-1, \\
K_{k}(t, x, y)&=\int_{0}^{t / 2} \int_{\mathbb{R}^{d}}\left|\partial_{u}^{k} T_{u}^\alpha(x, z)\right|_{u=t-s}|z|^{-d}\left|z\cdot \nabla_1 W_{s}^{\alpha}(z- y)\right| dz ds,\\
K_{k+1}(t, x, y)&=\int_{0}^{t / 2} \int_{\RR^{d}}\left|T_{s}^{\alpha}(x, z)\right||z|^{-d}\left|\left. z\cdot\nabla_1 \partial_{u}^{k} W_{u}^{\alpha}(z- y)\right|_{u=t-s}\right| dz ds.
\end{align*}

We consider, for every $j=0, \dots, k+1$ and $k \in \NN$,
\begin{equation*}
    \begin{split}
        \KK_{j,\loc}(x,y) = \int_{0}^{\infty} \chi_{\Omega_{\loc}}(x, y, t) K_j(t,x,y) t^{k-1} dt, \quad x, y \in \RR^{d}.
    \end{split}
\end{equation*}

Let $k \in \mathbb{N}$ and $j=0, \ldots, k-1$. According to \cite[Proposition~2.9]{BDM}, we have that, for
every  $\epsilon>0$,
\begin{equation}\label{eq: bound derivative T}
	\left|\partial_t^j T_{t/2}^{\alpha}(x,z)\right|\lesssim t^{-j-\frac d\alpha}\left(\frac{t^{1/\alpha}+|x-z|}{t^{1/\alpha}}\right)^{-d-\alpha+\epsilon}\left(1+\frac{t^{1/\alpha}}{|z|}\right)^{d-\beta}, \quad x,z\in \RR^d, t>0.
\end{equation}

By proceeding as in the proof of \cite[Lemma~2.11]{BDM},
\begin{equation}\label{eq: bound gradient-derivative W}
	\left|\nabla_1 \partial_t^{k-1-j} W_{t/2}^\alpha (y-z)\right| \lesssim t^{-\frac{d+1}{\alpha}-(k-1-j)} \left(\frac{t^{1/\alpha}}{t^{1/\alpha}+|z-y|}\right)^{d+\alpha+1+\epsilon}, \quad y,z\in \RR^d, t>0,
\end{equation}
for $\epsilon>0$. The constants appearing in \eqref{eq: bound derivative T} and \eqref{eq: bound gradient-derivative W} depend on $\epsilon$.

By \eqref{eq: bound derivative T} and \eqref{eq: bound gradient-derivative W} we get
\begin{align*}
	K_j(t,x,y)&\lesssim t^{-(k-1)-\frac{2d+1}{\alpha}} \int_{\RR^d} \left(1+\frac{t^{1/\alpha}}{|z|}\right)^{d-\beta}|z|^{1-d}\left(\frac{t^{1/\alpha}+|x-z|}{t^{1/\alpha}}\right)^{-d-\alpha+\epsilon_1}\\
	&\quad \times\left(\frac{t^{1/\alpha}+|z-y|}{t^{1/\alpha}}\right)^{-d-\alpha-1+\epsilon_2}dz, \quad x,y\in \RR^d, t>0,
\end{align*}
where $\epsilon_1, \epsilon_2>0$.
Note that, for a suitable positive function $F$,
\[
 t^{k} K_{j}(t, x, y)\leq t^{-\frac d\alpha} F\left(\frac{x}{t^{1/\alpha}}, \frac{y}{t^{1/\alpha}}\right), \quad x, y \in \mathbb{R}^d, t>0.
\]

By taking into account the last homogeneity property we are going to estimate $K_{j}(t, x, y)$, when $1 \leq|y|^{\alpha}$ and $|x-y| \leq \frac{|x| \wedge|y|}{2}$. In order to do this we use some ideas from~\cite[p. 30]{BDM}. We have that, for every $x,y\in\RR^d$, 
\begin{equation*}
    (1-t)^{-\frac{d+1}{\alpha}}\left(
    \frac{t^{1/\alpha}}{(1-t)^{1/\alpha}+|x-y|}\right)^{d+\alpha+1-\epsilon_2}\sim \left(\frac{1}{1+|x-y|}\right)^{d-\alpha+1-\epsilon_2},
\end{equation*}
for $t\in [\frac12,\frac23]$,
and
\begin{equation*}
    \left(1 \wedge \frac{|z|}{t^{1 / \alpha}}\right)^{\beta-d} t^{-\frac d\alpha}\left(\frac{t^{1 / \alpha}}{t^{1 / \alpha}+|x-y|}\right)^{d+\alpha-\epsilon_{1}} \sim(1 \wedge|z|)^{\beta-d}\left(\frac{1}{1+|x-y|}\right)^{d+\alpha-\epsilon_1},
\end{equation*}
for $t\in [\frac12,\frac23]$.

It follows that 
\begin{equation*}
    \begin{split}
        K_j(1,x,y)& = \frac{1}{2} \int_{\frac12}^{\frac23}\int_{\RR^d} (1\wedge|z|)^{\beta-d} \left( \frac{1}{1+ |x-y|}\right)^{d+\alpha-\epsilon_1} |z|^{1-\alpha} \left(\frac{1}{1+|z-x|}\right)^{d+\alpha+1-\epsilon_2}dzds
        \\ & \lesssim  \int_{\frac12}^{\frac23}\int_{\RR^d} (1-s)^{-\frac{1+d}{\alpha}} 
        \left(1\wedge\frac{|z|}{(1-s)^{1/\alpha}}\right)^{\beta-d} \left( \frac{(1-s)^{1/\alpha}}{(1-s)^{1/\alpha}+ |x-y|}\right)^{d+\alpha-\epsilon_1} \\
        &\quad \times |z|^{1/\alpha} s^{-\frac d\alpha} \left(\frac{s^{1/\alpha}}{s^{1/\alpha}+|z-x|}\right)^{d+\alpha+1-\epsilon_2}dzds
        \\ & \lesssim  \int_{0}^{1}\int_{\RR^d} (1-s)^{-\frac{1+d}{\alpha}} 
        \left(1\wedge\frac{|z|}{(1-s)^{1/\alpha}}\right)^{\beta-d} \left( \frac{(1-s)^{1/\alpha}}{(1-s)^{1/\alpha}+ |x-y|}\right)^{d+\alpha-\epsilon_1}\\
        &\quad \times  |z|^{1/\alpha} s^{-\frac d\alpha} \left(\frac{s^{1/\alpha}}{s^{1/\alpha}+|z-x|}\right)^{d+\alpha+1-\epsilon_2}dzds
        \\ & \lesssim (|x|\vee |y|)^{1-\alpha} 
        \left( \frac{1}{1+|x-y|}\right)^{d+\alpha-\epsilon_1}, \quad |y|\geq 1 \text{ and } |x-y|\leq \frac{|x|\wedge|y|}{2}.
    \end{split}
\end{equation*}
Then, for $j=0,\dots, k-1$,
\begin{equation*}
    t^k K_j(t,x,y) \leq C t^{-d/\alpha} \left( \frac{|x|\vee |y|}{t^\alpha}\right)^{1-\alpha} \left(\frac{t^{1/\alpha}}{t^{1/\alpha}+|x-y|}\right)^{d+\alpha-\epsilon_1},
\end{equation*}
for $t>0$, $|y|\geq t^{1/\alpha}$ and $|x-y|\leq \frac{|x|\wedge |y|}{2}$.

On the other hand, for $j=k$, we have that
\begin{equation*}
\begin{split}
    K_k(t,x,y) &\lesssim \int_{0}^{t/2} \int_{\RR^d} (t-s)^{-(k+d/\alpha)} \left( \frac{(t-s)^{1/\alpha}}{(t-s)^{1/\alpha}+|x-z|}\right)^{d+\alpha-\epsilon_1} \left( 1 + \frac{(t-s)^{1/\alpha}}{|z|}\right)\\
    &\quad\times|z|^{1-\alpha}s^{-\frac{d+1}{\alpha}}
    \left( \frac{s}{s^{1/\alpha} + |z-y|}\right)^{d+1+\alpha} dzds,
\end{split}
\end{equation*}
for $x$, $y\in\RR^d$ and $t>0$. Since $t-s\sim t$ when $0<s<t/2$, we get
\begin{equation*}
    \begin{split}
        t^k K_k(t,x,y) & \lesssim \int_0^{t/2} \int_{\RR^d} (t-s)^{-d/\alpha} \left( \frac{(t-s)^{1/\alpha}}{(t-s)^{1/\alpha} + |x-z|}\right)^{d+\alpha-\epsilon_1} \left( 1 + \frac{(t-s)^{1/\alpha}}{|z|}\right)\\
        &\quad \times |z|^{1-\alpha} s^{-\frac{d+1}{\alpha}}
         \left( \frac{s}{s^{1/\alpha} + |z-y|}\right)^{d+1+\alpha} dzds,
        \\ & = \frac{1}{t} G\left(\frac{x}{t^{1/\alpha}}, \frac{y}{t^{1/\alpha}} \right),
    \end{split}
\end{equation*}
for $x$, $y\in\RR^d$ and $t>0$, and a certain positive function $G$ defined in $\RR^d \times \RR^d$. By \cite[Lemma~4.2]{BDM} we deduce that
\begin{equation*}
    K_k(1,x,y) \lesssim (|x| \vee |y|)^{1-\alpha} \left( \frac{1}{1 + |x-y|}\right)^{d + \alpha - \epsilon_1}, 
\end{equation*}
for $|y|\geq 1$ and $|x-y| \leq \frac{|x|\wedge|y|}{2}$. Hence
\begin{equation*}
    t^k K_k(t,x,y) \lesssim \left( \frac{|x|\vee |y|}{t^{1/\alpha}}\right)^{1-\alpha} \frac{1}{t^{d/\alpha}} \left( \frac{t^{1/\alpha}}{t^{1/\alpha} + |x-y|}\right)^{d+\alpha - \epsilon_1},
\end{equation*}
for $|y|\geq t^{1/\alpha}$ and $|x-y| \leq \frac{|x|\wedge|y|}{2}$. 

We also have that
\begin{align*}
    K_{k+1}(t,x,y) &\lesssim \int_0^{t/2} \int_{\RR^d} 
    \frac{1}{s^{d/\alpha}} \left( 1+ \frac{s^{1/\alpha}}{|z|}\right)^{d-\beta} \left( \frac{s^{1/\alpha}}{s^{1/\alpha} + |x-z|}\right)^{d+\alpha-\epsilon_1} \\
    &\quad \times|z|^{1-\alpha} (t-s)^{-\frac{d+1}{\alpha} - k} 
    \left( \frac{(t-s)^{1/\alpha}}{(t-s)^{1/\alpha} + |z-y|}\right)^{d+\alpha+1-\epsilon_2} dzds,
\end{align*}
for $x, y\in\RR^d$ and $t>0$, and
\begin{equation*}
    t^k K_{k+1}(t,x,y) \leq \frac{1}{t^{d/\alpha}} H\left( \frac{x}{t^\alpha}, \frac{y}{t^\alpha}\right), 
\end{equation*}
for $x$, $y\in \RR^d$, $t>0$, where $H$ is a positive function defined in $\RR^d\times\RR^d$. Therefore,
\begin{equation*}
    \begin{split}
        K_{k+1}(1,x,y) & \lesssim \int_0^{1/2} \int_{\RR^d} \frac{1}{s^{d/\alpha}} \left( 1+ \frac{s^{1/\alpha}}{|z|}\right)^{d-\beta} \left( \frac{s^{1/\alpha}}{s^{1/\alpha}+|x-z|}\right)^{d+\alpha-\epsilon_1} 
        \\ &\quad \times |z|^{1-\alpha} \frac{1}{(1-s)^{\frac{d+1}{\alpha}+k}} \left( \frac{(1-s)^{1/\alpha}}{(1-s)^{1/\alpha} + |z-y|}\right)^{d+\alpha+1-\epsilon_2} dzds 
        \\ & \lesssim \int_0^{1/2} \int_{\RR^d} \frac{1}{s^{d/\alpha}} \left( 1+ \frac{s^{1/\alpha}}{|z|}\right)^{d-\beta} \left( \frac{s^{1/\alpha}}{s^{1/\alpha}+|x-z|}\right)^{d+\alpha-\epsilon_1} 
        \\ & \quad \times|z|^{1-\alpha}\frac{1}{(1-s)^{\frac{d+1}{\alpha}}} \left( \frac{(1-s)^{1/\alpha}}{(1-s)^{1/\alpha} + |x-z|}\right)^{d+\alpha+1-\epsilon_2} dzds,
    \end{split}
\end{equation*}
for $x$, $y\in \RR^d$. Again, according to \cite[Lemma~4.2]{BDM} we obtain
\begin{equation*}
    K_{k+1}(1,x,y) \lesssim (|x|\vee|y|)^{1-\alpha} \left( \frac{1}{1 + |x-y|}\right)^{d+\alpha - \epsilon_1},
\end{equation*}
for $|y|\geq 1$, $|x-y|\leq (|x|\wedge |y|)/2$, and
\begin{equation*}
    t^k K_{k+1}(t,x,y) \lesssim \left(\frac{|x|\vee|y|}{t^{1/\alpha}}\right)^{1-\alpha} \frac{1}{t^{d/\alpha}} \left(\frac{t^{1/\alpha}}{t^{1/\alpha}+|x-y|}\right)^{d+\alpha-\epsilon_1},
\end{equation*}
for $|y|\geq t^{1/\alpha}$ and $|x-y| \leq (|x|\wedge|y|)/2$. We conclude that, for every $j=0,\dots,k+1$, 
\begin{equation}\label{eq: bound Kloc}
    \mathbb{K}_{j,\loc}(x,y) \leq C \int_{0}^{\infty} \chi_{\Omega_{\loc}}(x,y,t) \left( \frac{|x|\vee|y|}{t^{1/\alpha}}\right)^{1-\alpha} 
    \frac{1}{t^{1/\alpha}}
 \left(\frac{t^{1/\alpha}}{t^{1/\alpha}+|x-y|}\right)^{d+\alpha-\epsilon_1} \frac{dt}{t},
 \end{equation}
 for $x, y\in\RR^d$. 

We consider now the global kernel, for $\ell \in \NN\cup\{0\}$, given by
\[\GG_\ell (x,y) =\int_0^\infty \chi_{\Omega_\glob} (x,y,t)\left|t^\ell \partial_t^{\ell+1}\left(W_t^\alpha(x,y)-T_t^\alpha(x,y)\right)\right|dt, \quad x,y\in \RR^d. \]
Let $0<\epsilon<\alpha$. By \cite[Proposition~2.9]{BDM} we get
\[\left|t^\ell \partial_t^{\ell+1} T_t^\alpha(x,y)\right|\lesssim t^{-\frac d\alpha-1} \left(\frac{t^{1/\alpha}}{t^{1/\alpha}+|x-y|}\right)^{d+\alpha-\epsilon}\left(1+\frac{t^{1/\alpha}}{|y|}\right)^{d-\beta}, \quad x,y\in \RR^d, t>0,\]
and by \cite[(2.13)]{BDM},
\[\left|t^\ell \partial_t^{\ell+1} W_t^\alpha(x,y)\right|\lesssim t^{-\frac d\alpha-1} \left(\frac{t^{1/\alpha}}{t^{1/\alpha}+|x-y|}\right)^{d+\alpha-\epsilon}, \quad x,y\in \RR^d, t>0.\]

Using \cite[p.~29, line~$-3$]{BDM}, we deduce that
\[\left|t^\ell \partial_t^{\ell+1} \left(W_t^\alpha(x,y)-T_t^\alpha(x,y)\right)\right|\lesssim \frac1t L_t^{\alpha-\epsilon}(x,y), \quad (x,y,t)\in \Omega_\glob,\]
where, for any $\gamma >0$
\begin{align*}
	L_t^\gamma(x,y)&=\chi_{\{|y|\leq t^{1/\alpha}\}}(x,y) \left(\frac{|y|}{t^{1/\alpha}}\right)^{\beta-d-(\alpha-\gamma)}t^{-\frac d\alpha} \left(\frac{t^{1/\alpha}}{t^{1/\alpha}+|x-y|}\right)^{d+\gamma}\\
	&\quad +\chi_{\{|y|\leq t^{1/\alpha}, |x|\sim |y|\}}(x,y) \left(\frac{|y|}{t^{1/\alpha}}\right)^{1+\beta-d-\alpha}t^{-\frac d\alpha}\\
	&\quad +\chi_{\left\{|y|\leq t^{1/\alpha}, |x-y|\geq \frac{|x|\wedge |y|}{2}\right\}}(x,y)t^{-\frac d\alpha} \left(\frac{t^{1/\alpha}}{t^{1/\alpha}+|x-y|}\right)^{d+\gamma}, \quad x,y\in \RR^d, t>0.
\end{align*}
Then,
\begin{equation}\label{eq: bound Gl}
	\GG_\ell(x,y)\lesssim \int_0^\infty \chi_{\Omega_\glob}(x,y,t) L_t^{\alpha-\epsilon}(x,y) \frac{dt}{t}, \quad x,y\in \RR^d.
\end{equation}

By combining \eqref{eq: bound Kloc} and \eqref{eq: bound Gl} and applying the Schur test as in \cite[pp.~33-35]{BDM}, we conclude that the operator $\TT_\ell$ is bounded on $L^p(\RR^d)$ for every $1\vee \frac{d}{\beta}<p<\infty$.

Therefore, the $\rho$-variation operator $V_\rho\left(\{t^k\partial_t^k \left(T_t^\alpha-W_t^\alpha\right)\}_{t>0}\right)$ is also bounded on $L^p(\RR^d)$ for every $1\vee \frac{d}{\beta}<p<\infty$.
\end{proof}

Since, as it was mentioned, the $\rho$-variation operator $V_\rho(\{t^k\partial _t^kW_t^\alpha\}_{t>0})$ is bounded on $L^p(\mathbb{R}^d)$ for every $1<p<\infty$, we conclude that the $\rho$-variation operator $V_\rho(\{t^k\partial_t^kT_t^\alpha \}_{t>0})$ is bounded on $L^p(\mathbb{R}^d)$ provided that $1\vee \frac{d}{\beta}<p<\infty$.

\section{Proof of Theorem~\ref{thm: var osc strong-weak}. Endpoint results}\label{sec: endpoint var}

We denote by $\sigma=1\vee \frac d\beta$. Let $f\in L^\sigma (\RR^d)\cap L^2(\RR^d)$. We can write
\begin{equation}\label{eq: T small than variation}
    |T_t^\alpha(f)|\leq V_{\rho}\left(\{T_t^\alpha\}_{t>0}\right)(f)+|T_s(f)|, \quad 0<s<t<\infty.
\end{equation}
Since $V_{\rho}\left(\{T_t^\alpha\}_{t>0}\right)$ is bounded on $L^2(\RR^d)$, $V_{\rho}\left(\{T_t^\alpha\}_{t>0}\right)(f)<\infty$ almost everywhere. Then, there exists the limit $\lim_{s\to 0^+} T_s^\alpha(f)(x)$ for a.e. $x\in \mathbb R^d$. 

On the other hand, $\{T_t^\alpha\}_{t>0}$ is a $C_0$-semigroup in $L^2(\RR^d)$. Then, $\lim_{s\to 0^+} T_s^\alpha(f)=f$ in $L^2(\RR^d)$, from which we get $\lim_{s\to 0^+} T_s^\alpha(f)(x)=f(x)$ for a.e. $x\in \mathbb R^d$. From \eqref{eq: T small than variation} we deduce that
\begin{equation}\label{eq: T small than variation 2}
    |T_t^\alpha(f)(x)|\leq V_{\rho}\left(\{T_t^\alpha\}_{t>0}\right)(f)(x)+|f(x)|, \quad t>0, \text{ a.e }x\in \mathbb R^d.
\end{equation}

We define the maximal operator $T_*^{\alpha}$ by
\[T_*^{\alpha}(f)=\sup_{t>0} |T_t^\alpha(f)|.\]
From \eqref{eq: T small than variation 2} we get
\[T_*^{\alpha}(f)\leq V_{\rho}\left(\{T_t^\alpha\}_{t>0}\right)(f)+|f|.\]
Clearly, the maximal operator $T_*^{\alpha}$ is bounded from $L^\sigma (\RR^d)$ to $L^{\sigma,\infty} (\RR^d)$ provided that $V_{\rho}\left(\{T_t^\alpha\}_{t>0}\right)$ satisfies the same boundedness. Therefore, in order to prove that, when $d>\beta$, the $\rho$-variation operator $V_\rho(\{T_t^\alpha\}_{t>0})$ is not bounded from $L^\sigma(\RR^d)$ to $L^{\sigma,\infty}(\RR^d)$ we are going to see that $T_*^{\alpha}$ is not bounded from $L^\sigma (\RR^d)$ to $L^{\sigma,\infty} (\RR^d)$.

Assume that $d>\beta$, that is, $\sigma>1$. We consider $f\in L^2(\RR^2)$, $f\geq 0$. According to \cite[(2.12) and (2.18)]{BDM}, for any $x\in \RR^d$ and $t>0$ we deduce that
\begin{align*}
    T_t^\alpha(f)(x)&\geq C\frac{1}{t^{1/\alpha}} \int_{\RR^d} \left(\frac{t^{1/\alpha}}{t^{1/\alpha}+|x-y|}\right)^{d+\alpha} \left(1\wedge \frac{|y|}{t^{1/\alpha}}\right)^{\beta-d} f(y) dy\\
    &\geq C\frac{1}{t^{1/\alpha}} \int_{B(0,|x|)} \left(\frac{t^{1/\alpha}}{t^{1/\alpha}+|x-y|}\right)^{d+\alpha} \left(1\wedge \frac{|y|}{t^{1/\alpha}}\right)^{\beta-d} f(y) dy.
\end{align*}
Since $|x-y|\leq 2|x|$ when $y\in B(0,|x|)$, we have
\begin{align}\label{eq: lower bound maximal}
    T_*^{\alpha}(f)(x)&\geq T_{|x|^\alpha}^{\alpha}(f)(x)\geq \frac{C}{|x|^d} \int_{B(0,|x|)} \left(\frac{|x|}{|x|+|x-y|}\right)^{d+\alpha} \left(1\wedge \frac{|y|}{|x|}\right)^{\beta-d} f(y) dy \nonumber \\
    &\geq \frac{C}{|x|^d} \int_{B(0,|x|)} \left(\frac{|y|}{|x|}\right)^{\beta-d} f(y) dy, \quad x\in \RR^d\setminus\{0\}.
\end{align}

We consider the function $g(x)=|x|^{\beta-d}$, $x\in \RR^d\setminus\{0\}$. We can see that $g\notin L^{\frac{d}{d-\beta}}(B(0,1))$. Thus, for every $n\in \NN$ we can find a positive measurable function $g_n\in L^{\frac{d}{\beta}}(B(0,1))\cap L^{2}(B(0,1))$ such that $\|g_n\|_{L^{\frac{d}{\beta}}(B(0,1))}\leq 1$ and $\int_{B(0,1)} |y|^{d-\beta} g_n(y) dy>n$. We define $g_n(y)=0$ for $|y|\geq 1$ and $n\in \NN$.

By \eqref{eq: lower bound maximal}, for every $n\in \NN$ we get
\[T_*^{\alpha}(g_n)(x)\geq \frac{C}{|x|^\beta}\int_{B(0,1)}g_n(x)|y|^{\beta-d} dy\geq C\frac{n}{|x|^\beta}, \quad x\in \RR^d\setminus\{0\}.\]
This yields 
\[T_*^{\alpha}(g_n)(x)\geq C(n-1), \quad 1<|x|<\left(\frac{n}{n-1}\right)^{1/\beta}, n\in \NN, \, n>1.\]
Then
\begin{align*}
    \left|\left\{x\in \RR^d: |T_*^\alpha(g_n)(x)|>C\frac{n-1}{2}\right\}\right|&\geq \int_{B\left(0,\left(\frac{n}{n-1}\right)^{1/\beta}\right)\setminus B(0,1)} dx\\
    &=|B(0,1)|\left(\left(\frac{n}{n-1}\right)^{d/\beta}-1\right), \quad n\in \NN,\, n>1.
\end{align*}

On the other hand, 
\[\left(\frac{\|g_n\|_{d/\beta}}{n-1}\right)^{d/\beta}\leq \frac{1}{(n-1)^{d/\beta}}, \quad n\in \NN,\, n>1.\]
Since the sequence $\{n^{d/\beta}-(n-1)^{d/\beta}\}_{n> 1}$ is not bounded, $T_*^{\alpha}$ is not bounded from $L^{d/\beta}(\RR^d)$ to $L^{d/\beta,\infty}(\RR^d)$.

We are going to see that the $\rho$-variation operator $V_\rho\left(\left\{t^k \partial_t^k \left(T_t^\alpha-W_t^\alpha\right)\right\}_{t>0}\right)$ is bounded from $L^1(\RR^d)$ to $L^{1,\infty}(\RR^d)$, when $\beta\ge d$. The following result actually holds for every $\rho\ge 1$.

\begin{prop}\label{prop: weak difference} Let $\alpha\in (1,2\wedge (d+2)/2)$, $\beta\in [d,d+\alpha)$, and $\kappa=\Phi(\beta)$, $k\in\NN\cup \{0\}$ and $\rho\geq 1$. The $\rho$-variation operator $V_{\rho}\left(\left\{t^{k} \partial_{t}^{k}\left(W_{t}^{\alpha}-T_t^\alpha)\right)\right\}_{t>0}\right)$ is bounded from $L^1(\RR^d)$ into $L^{1,\infty}(\RR^d)$.
\end{prop}

\begin{proof} As in Section~\ref{sec: difference} we consider the sets
\[\Omega_\loc =\left\{(x,y,t)\in \RR^d\times \RR^d\times (0,\infty): |y|\geq t^{1/\alpha}\text{ and }|x-y|\leq \frac{|x|\wedge|y|}{2}\right\}\]
and $\Omega_\glob =\left(\RR^d\times \RR^d\times (0,\infty)\right)\setminus \Omega_\loc$.

Also, for $\ell=0, 1, \dots, k$, we define the operators
\[\TT _{\ell, \loc}(f)(x)=\int_{\RR^d} f(y)\KK _{\ell,\loc}(x,y) dy, \quad x\in \RR^d\]
and 
\[\TT _{\ell, \glob}(f)(x)=\int_{\RR^d} f(y)\KK _{\ell,\glob}(x,y) dy, \quad x\in \RR^d,\]
where, for $x,y\in \RR^d$ and $t>0$,
\[\KK _{\ell,\loc}(x,y)=\int_0^\infty \chi_{\Omega_\loc}(x,y,t)\left|t^\ell \partial_t^{\ell+1}\left(T_t^\alpha(x,y)-W_t^\alpha(x,y)\right)\right| dt\]
and 
\[\KK _{\ell,\glob}(x,y)=\int_0^\infty \chi_{\Omega_\glob}(x,y,t)\left|t^\ell \partial_t^{\ell+1}\left(T_t^\alpha(x,y)-W_t^\alpha(x,y)\right)\right| dt.\]

Fix $\ell=0,1, \dots, k$. We shall first study the operator $\TT _{\ell, \glob}$. According to \cite[Proposition~2.9]{BDM} and \cite[Proposition 2.2]{BDQ}, since $\beta<d$ we deduce that
\[|t^\ell \partial_t^{\ell+1}\left(T_t^\alpha(x,y)-W_t^\alpha(x,y)\right)|\lesssim \frac{1}{t^{d/\alpha+1}}\left(\frac{t^{1/\alpha}}{t^{1/\alpha}+|x-y|}\right)^{d+\alpha-\epsilon}\]
for $x,y\in \RR^d$, $t>0$ and $0<\epsilon<1$. Therefore,
\begin{align*}
    \KK _{\ell,\glob}(x,y)&\lesssim \int_{|y|^{\alpha}}^\infty \frac{1}{t^{d/\alpha+1}} dt \chi_{\left\{|x-y|\leq \frac{|x|\wedge|y|}{2}\right\}}(x,y)\\
    &\quad +\int_0^\infty \frac{1}{t^{d/\alpha+1}} \left(\frac{t^{1/\alpha}}{t^{1/\alpha}+|x-y|}\right)^{d+\alpha-\epsilon} dt \chi_{\left\{|x-y|> \frac{|x|\wedge|y|}{2}\right\}}(x,y)\\
    &\lesssim \frac{1}{|y|^d} \chi_{\left\{|x-y|\leq \frac{|x|\wedge|y|}{2}\right\}}(x,y)+\frac{1}{|x-y|^d}\chi_{\left\{|x-y|> \frac{|x|\wedge|y|}{2}\right\}}(x,y)\\
    &:= \KK_{1,1}(x,y)+\KK_{1,2}(x,y), \quad x,y\in \RR^d.
\end{align*}

We consider, for $j=1,2$, the operators 
\[\TT_{1,j}(f)(x)=\int_{\RR^d} \KK_{1,j}(x,y) f(y)dy, \quad x\in \RR^d.\]
Since $|x-y|\leq \frac{|x|\wedge|y|}{2}$ implies $|x|/2<|y|<3|x|/2$, we get 
\[\int_{\RR^d} \KK_{1,1}(x,y) dy\lesssim \int_{|x|/2}^{3|x|/2} \frac{1}{\blue{r}} d\blue{r} \sim 1, \quad x\in \RR^d\]
and
\[\int_{\RR^d} \KK_{1,1}(x,y) dx\lesssim \frac{1}{|y|^d}\int_{B(0,2|y|)\setminus B(0,2|y|/3)} dx \sim 1, \quad y\in \RR^d.\]
Consequently, the operator $\TT_{1,1}$ is bounded on $L^p(\RR^d)$ for every $1\leq p\leq \infty$.

On the other hand,
\begin{align*}
    \KK_{1,2}(x,y)&\lesssim \frac{1}{|x-y|^d} \left(\chi_{\{|x|\geq 3|y|/2\}}(x,y)+\chi_{\{|y|/2<|x|<3|y|/2\}}(x,y)\right)\chi_{\{|y|< |x|\wedge 2|x-y|}\}(x,y)\\
    &\quad + \frac{1}{|x-y|^d} \left(\chi_{\{|y|\geq 3|x|/2\}}(x,y)+\chi_{\{|x|/2<|y|<3|x|/2\}}(x,y)\right)\chi_{\{|x|< |y|\wedge 2|x-y|}\}(x,y)\\
    &\lesssim \frac{1}{|x-y|^d} \left(\chi_{\{|x|> 3|y|/2\}}(x,y)+\chi_{\{|y|>3|x|/2\}}(x,y)\right)\\
    &\quad + \frac{1}{|x-y|^d}\left(\chi_{\{|y|< |x|\wedge 2|x-y|}\}(x,y)+\chi_{\{|x|< |y|\wedge 2|x-y|}\}(x,y)\right)\chi_{\{|y|/2<|x|<3|y|/2\}}(x,y)\\
    &\lesssim \frac{1}{|x|^d}\chi_{\{|x|>|y|\}}(x,y)+\frac{1}{|y|^d}\chi_{\{|y|>|x|\}}(x,y), \quad x,y\in \RR^d.
\end{align*}
It follows that
\begin{equation*}
    |\TT_{1,2}(f)(x)|\lesssim \frac{1}{|x|^d}\int_{B(x,2|x|)} |f(y)|dy+\int_{B(0,|x|)^c} \frac{|f(y)|}{|y|^d}dy =\TT_{1,2,1}(f)(x)+\TT_{1,2,2}(f)(x), \quad x\in \RR^d.
\end{equation*}
Since $\TT_{1,2,1}(f)\lesssim M_{\HL}(|f|)$, where $M_{\HL}$ denotes the centered Hardy-Littlewood maximal operator, $\TT_{1,2,1}$ is bounded from $L^1(\RR^d)$ to $L^{1,\infty}(\RR^d)$.

We also get
\begin{equation*}
    \int_{\RR^d} |\TT_{1,2,2}(f)(x)| dx\lesssim \int_{\RR^d} \int_{B(0,|x|)^c} \frac{|f(y)|}{|y|^d}dy dx\lesssim \int_{\RR^d} \frac{|f(y)|}{|y|^d} \int_{|x|<|y|} dx dy\lesssim \int_{\RR^d}|f(y)| dy
\end{equation*}
for every $f\in L^1(\RR^d)$. Hence, $\TT_{1,2,2}$ is bounded on $L^1(\RR^d)$.

We conclude that $T_{\ell,\glob}$ is bounded from $L^1(\RR^d)$ into $L^{1,\infty}(\RR^d)$.

We will now study the operator $T_{\ell,\loc}$. As it was proved in Section~\ref{sec: difference}, we have that
\begin{align*}
|T_{\ell,\loc}(f)(x)|
    &\lesssim \int_{\RR^d} |f(y)| \int_0^\infty \frac{1}{t^{d/\alpha+1}} \left(\frac{|x|\vee |y|}{t^{1/\alpha}}\right)^{1-\alpha} \left(\frac{t^{1/\alpha}}{t^{1/\alpha}+|x-y|}\right)^{d+\alpha-\epsilon} \chi_{\Omega_\loc} (t,x,y) dt dy\\
    &\lesssim \int_{B(x,|x|/2)} |f(y)| \int_0^\infty \frac{1}{t^{d/\alpha+1}} \left(\frac{|x|}{t^{1/\alpha}}\right)^{1-\alpha} \left(\frac{t^{1/\alpha}}{t^{1/\alpha}+|x-y|}\right)^{d+\alpha-\epsilon} dt dy\\
    &\lesssim \int_{B(x,|x|/2)} |f(y)| \frac{|x|^{1-\alpha}}{|x-y|^{d+1-\alpha}}\int_0^\infty \frac{1}{u^{(d+1)/\alpha}} \left(\frac{u^{1/\alpha}}{u^{1/\alpha}+1}\right)^{d+\alpha-\epsilon} du dy\\
    &\lesssim \int_{B(x,|x|/2)} |f(y)| \frac{|x|^{1-\alpha}}{|x-y|^{d+1-\alpha}} dy, \quad x\in \RR^d,
\end{align*}
provided that $\epsilon$ is small enough.

It is easy to see that
\[\int_{B(x,|x|/2)}  \frac{|x|^{1-\alpha}}{|x-y|^{d+1-\alpha}} dy\sim |x|^{1-\alpha} \int_{0}^{|x|/2} r^{\alpha-2} dr \sim 1, \quad x\in \RR^d,\]
and
\[\int_{B(y,3|y|/4)}  \frac{|y|^{1-\alpha}}{|x-y|^{d+1-\alpha}} dx\sim 1, \quad y\in \RR^d.\]
Then, by taking into account that  $2|y|/3\le|x|\le 3|y|/2$ provided that $(x,y,t)\in \Omega_{\loc}$, with $t>0$, the operator $T_{\ell,\loc}$ is bounded on $L^p(\RR^d)$ for every $1\leq p\leq \infty$. By combining all the above estimates we conclude that the operator $V_\rho\left(\left\{t^k \partial_t^k \left(T_t^\alpha-W_t^\alpha\right)\right\}_{t>0}\right)$ is bounded from $L^1(\RR^d)$ to $L^{1,\infty}(\RR^d)$.
\end{proof}

According to Proposition~\ref{prop: weak type 1-1}, the operator $V_\rho\left(\left\{t^k \partial_t^k W_t^\alpha\right\}_{t>0}\right)$ is bounded from $L^1(\RR^d)$ to $L^{1,\infty}(\RR^d)$ and thus $V_\rho\left(\left\{t^k \partial_t^k T_t^\alpha\right\}_{t>0}\right)$ is bounded from $L^1(\RR^d)$ to $L^{1,\infty}(\RR^d)$.

\section{Proof of Theorem~\ref{thm: var osc strong-weak} for the oscillation operators}\label{sec: endpoint osc}

We proceed as in the case of $\rho$-variation operators. Let $\{t_j\}_{j=1}^\infty$ be a decreasing sequence of positive numbers such that $\lim_{j\to \infty} t_j=0$, and let $k\in \NN$. We have that
\begin{align*}
	O&\left(\left\{t^k\partial_t^k\left(T_t^\alpha-W_t^\alpha\right)\right\}_{t>0},\{t_j\}_{j=1}^\infty\right)(f)\\
	&=\left(\sum_{j\in \NN} \sup_{t_{j+1}\leq \epsilon_{j+1}<\epsilon_j\leq t_j} \left|\left.t^k\partial_t^k\left(T_t^\alpha-W_t^\alpha\right)(f)\right|_{t=\epsilon_{j+1}}-\left.t^k\partial_t^k\left(T_t^\alpha-W_t^\alpha\right)(f)\right|_{t=\epsilon_j}\right|^2\right)^{1/2}\\
	&=\left(\sum_{j\in \NN} \sup_{t_{j+1}\leq \epsilon_{j+1}<\epsilon_j\leq t_j} \left|\int_{\epsilon_{j+1}}^{\epsilon_j} \partial_t \left(t^k\partial_t^k\left(T_t^\alpha-W_t^\alpha\right)(f)\right)dt\right|^2\right)^{1/2}\\
	&\leq \left(\sum_{j\in \NN} \left(\int_{t_{j+1}}^{t_j} \left|\partial_t \left(t^k\partial_t^k\left(T_t^\alpha-W_t^\alpha\right)(f)\right)\right|dt\right)^2\right)^{1/2}\\ 
	&\leq \sum_{j\in \NN} \int_{t_{j+1}}^{t_j} \left|\partial_t \left(t^k\partial_t^k\left(T_t^\alpha-W_t^\alpha\right)(f)\right)\right|dt\\ 
	&\leq \int_0^\infty \left|\partial_t \left(t^k\partial_t^k\left(T_t^\alpha-W_t^\alpha\right)(f)\right)\right|dt.
\end{align*}
Then, as in the case of the $\rho$-variation operators, we can prove that the oscillation operator of the difference $O\left(\left\{t^k\partial_t^k\left(T_t^\alpha-W_t^\alpha\right)\right\}_{t>0},\{t_j\}_{j=1}^\infty\right)$ is bounded on $L^p(\RR^d)$ for every $1\vee\frac d\beta<p<\infty$ and from $L^1(\RR^d)$ into $L^{1,\infty}(\RR^d)$, when $d\le \beta$.

As it was mentioned in Section~\ref{sec: intro}, from \cite[Corollary~6.2 and comment on the top of p.~2092]{LeMX1}, we deduce that the oscillation operator $O\left(\left\{t^k\partial_t^kW_t^\alpha\right\}_{t>0},\{t_j\}_{j=1}^\infty\right)$ is bounded on $L^p(\RR^d)$, for $1<p<\infty$. Also, according to Proposition~\ref{prop: weak type 1-1}, $O\left(\left\{t^k\partial_t^kW_t^\alpha\right\}_{t>0},\{t_j\}_{j=1}^\infty\right)$ is bounded from $L^1(\RR^d)$ into $L^{1,\infty}(\RR^d)$. 

Finally, we conclude that $O\left(\left\{t^k\partial_t^k T_t^\alpha\right\}_{t>0},\{t_j\}_{j=1}^\infty\right)$ is also bounded on $L^p(\RR^d)$ for every $1\vee\frac d\beta<p<\infty$ and from $L^1(\RR^d)$ into $L^{1,\infty}(\RR^d)$ when $d\le \beta$.

\section{Proof of Theorem~\ref{thm: jump strong-weak}}\label{sec: jump}

It is known (\cite[p.~6712]{JSW}) that, for every $\lambda>0$ and $\rho>2$,
\begin{equation}\label{eq: jump<var}
\lambda \left[J\left(\{t^k\partial_t^k T_t^\alpha\}_{t>0},\lambda\right)(f)\right]^{1/\rho}\leq 2^{1+\frac1\rho}V_\rho\left(\{t^k\partial_t^k T_t^\alpha\}_{t>0}\right)(f).
\end{equation}
From the boundedness property of the $\rho$-variation operators $V_\rho\left(\{t^k\partial_t^k T_t^\alpha\}_{t>0}\right)$, we deduce that the operator $\lambda \left[J\left(\{t^k\partial_t^k T_t^\alpha\}_{t>0},\lambda\right)\right]^{1/\rho}$ is bounded in $L^p(\RR^d)$, uniformly in $\lambda>0$, for every $1\vee \frac{d}{\beta}<p<\infty$ and from $L^1(\RR^d)$ into $L^{1,\infty}(\RR^d)$ when $d\le \beta$.

Since $\{W_t^\alpha\}_{t>0}$ is subordinated to the classical heat semigroup, it is a diffusion symmetric semigroup in the sense of Stein (\cite{StLP}). Then, according to \cite[(3)]{MSZ}, the family $\lambda \left[J\left(\{W_t^\alpha\}_{t>0}, \lambda\right)\right]^{1/2}$ is bounded in $L^p(\RR^d)$ for every $1<p<\infty$.

On the other hand, for every $\lambda>0$, 
\begin{align*}
	\lambda \left[J\left(\{T_t^\alpha\}_{t>0}, \lambda\right)(f)\right]^{1/2}&\lesssim \lambda \left[J\left(\{T_t^\alpha-W_t^\alpha\}_{t>0}, \tfrac{\lambda}{2}\right)(f)\right]^{1/2}+\lambda \left[J\left(\{W_t^\alpha\}_{t>0}, \tfrac{\lambda}{2}\right)(f)\right]^{1/2}\\
	&\lesssim V_2\left(\{T_t^\alpha-W_t^\alpha\}_{t>0}\right)(f)+\lambda \left[J\left(\{W_t^\alpha\}_{t>0}, \tfrac{\lambda}{2}\right)(f)\right]^{1/2}.
\end{align*}
According to Proposition~\ref{prop: difference} we deduce that the operator
$\{\lambda \left[J\left(\{T_t^\alpha\}_{t>0}, \lambda\right)\right]^{1/2}$ is bounded in $L^p(\RR^d)$, uniformly in $\lambda>0$, provided that $1\vee \frac{d}{\beta}<p<\infty$.

\section{Weighted inequalities}\label{sec: weighted}

We are going to prove the result for the $\rho$-variation operators $V_\rho\left(\{t^k\partial_t^k T_t^\alpha\}_{t>0}\right)$.  The same arguments work for the oscillation operator $O\left(\{t^k\partial_t^k T_t^\alpha\}_{t>0}, \{t_j\}_{j\in \NN}\right)$. The result for the jump operator can be deduced from the one for the $\rho$-variation operator by using (\ref{eq: jump<var}).

According to Theorem~\ref{thm: var osc strong-weak}\ref{itm: var osc strong}, the $\rho$-variation operator $V_\rho\left(\{t^k\partial_t^k T_t^\alpha\}_{t>0}\right)$ is bounded on $L^p(\RR^d)$ provided that $1\vee \frac{d}{\beta}<p<\infty$.

In order to prove Theorem~\ref{thm: weighted Lp}, we shall use \cite[Theorem~6.6]{BZ} (see also \cite[Theorem~2.7]{Na23}) and we adapt the proof of \cite[Proposition~3.5]{Na23}.

Let $1\vee \frac{d}{\beta}<p_0<q_0<\infty$. If $B$ is a ball in $\RR^d$ we define, for every $j\in \NN$, the annuli $S_j(B)=2^jB\setminus 2^{j-1}B$, and $S_0(B)=B$.

We consider $m\in \NN$ to be fixed later, and define, for every $t>0$, the operator
\[A_t=(I-T_t^\alpha)^m.\]
Our objective is to prove the following properties for $A_t$.

\begin{lem}\label{lem: aux weighted}
	For every ball $B=B(x_B, r_B)$ with $x_B\in \RR^d$ and $r_B>0$, and $j\in \NN\cup\{0\}$, there exist numbers $c_j(B)>0$ such that
	\begin{equation}\label{eq: cj sumables}
		\sup_{B} \sum_{j=0}^\infty c_j(B)2^{\alpha j}<\infty, 
	\end{equation} 
	and, for each function $f$ supported in $B$,
	\begin{enumerate}[label=(\roman*)]
		\item \label{itm: aux weighted i}for every $j\in \NN$, $j\geq 2$,
		\[\left(\frac{1}{|S_j(B)|}\int_{S_j(B)} \left|V_\rho\left(\{t^k\partial_t^k T_t^\alpha\}_{t>0}\right)(A_{r_B^\alpha}f)(x)\right|^{q_0}dx\right)^{1/q_0}\leq c_j(B)\left(\frac{1}{|B|}\int_B |f(x)|^{q_0}dx\right)^{1/q_0}.\]
		\item \label{itm: aux weighted ii} for every $j\in \NN\cup \{0\}$,
		\[\left(\frac{1}{|S_j(B)|}\int_{S_j(B)} \left|(I-A_{r_B^\alpha})(f)(x)\right|^{q_0}dx\right)^{1/q_0}\leq c_j(B)\left(\frac{1}{|B|}\int_B |f(x)|^{p_0}dx\right)^{1/p_0}.\]
	\end{enumerate}
\end{lem} 

\begin{proof}
	 We will show first \ref{itm: aux weighted ii}. Let $B=B(x_B,r_B)$ be a ball in $\RR^d$ and note that
	\[I-A_{r_B^\alpha}=I-(I-T_{r_B^\alpha}^\alpha)^m=\sum_{k=1}^m \binom{m}{k} T_{kr_B^\alpha}^\alpha.\]
	
	Suppose $f$ is a function supported on $B$, and let $j\in \NN$, $j\geq 2$. According to \cite[(2.12) and Theorem~2.2]{BDM}, we get
	\begin{align*}
		&\left(\frac{1}{|S_j(B)|}\int_{S_j(B)} \left|T_{kr_B^\alpha}^\alpha(f)(x)\right|^{q_0}dx\right)^{1/q_0}\\
		&\qquad\qquad\lesssim \left( \left(\frac{1}{k^{1/\alpha}}\right)^{d/p_0}\vee \left(\frac{1}{k^{1/\alpha}}\right)^{d}\right) \left(1+\frac{2^j}{k^{1/\alpha}}\right)^{-d-\beta} \left(\frac{1}{|B|}\int_B |f(x)|^{p_0}dx\right)^{1/p_0}\\
		&\qquad\qquad\lesssim 2^{-j(d+\beta)} \left(\frac{1}{|B|}\int_B |f(x)|^{p_0}dx\right)^{1/p_0}, \quad k=1,\dots, m.
	\end{align*}
	Then, for every $j\geq 2$, 
	\[\left(\frac{1}{|S_j(B)|}\int_{S_j(B)} \left|(I-A_{r_B^\alpha})(f)(x)\right|^{q_0}dx\right)^{1/q_0}\lesssim 2^{-j(d+\beta)}\left(\frac{1}{|B|}\int_B |f(x)|^{p_0}dx\right)^{1/p_0}.\]
	
    Also, for $j=0, 1$, we can follow the proof of \cite[Theorem~2.2]{BDM} to get that
    \begin{align*}
        \left(\frac{1}{|S_j(B)|}\int_{S_j(B)} \left|T_{kr_B^\alpha}^\alpha(f)(x)\right|^{q_0}dx\right)^{1/q_0}&\lesssim  \left( \left(\frac{1}{k^{1/\alpha}}\right)^{d/p_0}\vee \left(\frac{1}{k^{1/\alpha}}\right)^{d}\right) \left(\frac{1}{|B|}\int_B |f(x)|^{p_0}dx\right)^{1/p_0}\\
        &\lesssim \left(\frac{1}{|B|}\int_B |f(x)|^{p_0}dx\right)^{1/p_0}.
    \end{align*}
    Hence,
	\[\left(\frac{1}{|S_j(B)|}\int_{S_j(B)} \left|(I-A_{r_B^\alpha})(f)(x)\right|^{q_0}dx\right)^{1/q_0}\lesssim  \left(\frac{1}{|B|}\int_B |f(x)|^{p_0}dx\right)^{1/p_0},\]
	which proves the desired estimate with $c_j(B)=2^{-j(d+\beta)}$. It clearly verifies \eqref{eq: cj sumables} since $\alpha-\beta-d<0$. 
	
	Let us now prove \ref{itm: aux weighted i}. Consider $j\in \NN$, $j\geq 2$, $B$ and $f$ as before. We define, for $\ell\in \NN$, the operator
	\[\mathcal{S}_\ell (f)(x)=\int_0^\infty t^{\ell-1} \left|\partial_t^\ell T_t^\alpha(f)(x)\right| dt, \quad x\in \RR^d.\]
	
	It will suffice to show that, for some $\gamma>0$,
	\begin{equation}\label{eq: bound Sl}
		\left(\frac{1}{|S_j(B)|}\int_{S_j(B)} \left|\mathcal{S}_\ell (A_{r_B^\alpha}(f))(x)\right|^{q_0}dx\right)^{1/q_0}\lesssim 2^{-j(d+\gamma)} \left(\frac{1}{|B|}\int_B |f(x)|^{q_0}dx\right)^{1/q_0}.
	\end{equation}
	
	By Minkowski's inequality, we split the right-hand side as follows
	\begin{align*}
		\frac{1}{|S_j(B)|}\int_{S_j(B)} \left|\mathcal{S}_\ell (A_{r_B^\alpha}(f))(x)\right|^{q_0}dx&\lesssim \int_0^\infty \left(	\frac{1}{|S_j(B)|}\int_{S_j(B)} |t^{\ell-1}\partial_t^\ell T_t^\alpha(A_{r_B^\alpha} f)(x)|^{q_0}dx\right)^{1/q_0} dt\\
		&\lesssim \int_0^{r_B^\alpha} \left(	\frac{1}{|S_j(B)|}\int_{S_j(B)} |t^{\ell-1}\partial_t^\ell T_t^\alpha(A_{r_B^\alpha} f)(x)|^{q_0}dx\right)^{1/q_0} dt\\
		&\quad +\int_{r_B^\alpha}^\infty \left(	\frac{1}{|S_j(B)|}\int_{S_j(B)} |t^{\ell-1}\partial_t^\ell T_t^\alpha(A_{r_B^\alpha} f)(x)|^{q_0}dx\right)^{1/q_0} dt\\
		&:=F_1+F_2.
	\end{align*}
	
	We can write
	\begin{align*}
		F_1&=\int_0^{r_B^\alpha} \left(	\frac{1}{|S_j(B)|}\int_{S_j(B)} \left|t^{\ell-1}\int_0^{r_B^\alpha}\cdots \int_0^{r_B^\alpha}\partial_t^{\ell+m} T_{t+s_1+\dots+s_m}^\alpha(f)(x)ds_1\dots ds_m\right|^{q_0}dx\right)^{1/q_0} dt\\
		&\leq \int_0^{r_B^\alpha} \int_0^{r_B^\alpha}\cdots \int_0^{r_B^\alpha}t^{\ell-1}\left(\frac{1}{|S_j(B)|}\int_{S_j(B)}\left|\partial_t^{\ell+m} T_{t+s_1+\dots+s_m}^\alpha(f)(x) \right|^{q_0}dx\right)^{1/q_0} ds_1\dots ds_m dt. 
	\end{align*}
	
	Using \cite[Proposition~2.9 and Theorem~2.2]{BDM}, we get
	\begin{align*}
		F_1&\lesssim  \int_0^{r_B^\alpha} \int_0^{r_B^\alpha}\cdots \int_0^{r_B^\alpha} \frac{t^{\ell-1}}{(t+s_1+\dots+s_m)^{\ell+m}} \left(1+\frac{2^j r_B}{(t+s_1+\dots+s_m)^{1/\alpha}}\right)^{-d-\beta}\\
		&\quad \times \left(\frac{r_B}{(t+s_1+\dots+s_m)^{1/\alpha}}\right)^d ds_1\dots ds_m dt \left(\frac{1}{|B|}\int_B |f(x)|^{q_0}dx\right)^{1/q_0}\\
		&\lesssim \int_0^{r_B^\alpha} \int_0^{r_B^\alpha}\cdots \int_0^{r_B^\alpha} \frac{2^{-j(d+\beta)}r_B^{-\beta}}{(t+s_1+\dots+s_m)^{m+1-\beta/\alpha}} ds_1\dots ds_m dt \left(\frac{1}{|B|}\int_B |f(x)|^{q_0}dx\right)^{1/q_0}.
	\end{align*}
	
	Taking $m>\beta/\alpha-1$, we have
	\[(t+s_1+\dots+s_m)^{-(m+1-\beta/\alpha)}\leq t^{-\frac{m+1-\beta/\alpha}{m+1}}s_1^{-\frac{m+1-\beta/\alpha}{m+1}}\cdots s_m^{-\frac{m+1-\beta/\alpha}{m+1}}, \quad t,s_1, \dots, s_m>0.\]
	Hence,
	\begin{align*}
		F_1&\lesssim 2^{-j(d+\beta)}r_B^{-\beta} \left(\int_0^{r_B^\alpha}t^{-\frac{m+1-\beta/\alpha}{m+1}} dt\right)^{m+1} \left(\frac{1}{|B|}\int_B |f(x)|^{q_0}dx\right)^{1/q_0}\\
		&\lesssim 2^{-j(d+\beta)}\left(\frac{1}{|B|}\int_B |f(x)|^{q_0}dx\right)^{1/q_0}.
	\end{align*}
	
	We now study $F_2$. We can proceed in a similar way to obtain
	\begin{align*}
		F_2&\lesssim \int_{r_B^\alpha}^\infty \int_0^{r_B^\alpha}\cdots \int_0^{r_B^\alpha}\frac{t^{\ell-1}}{(t+s_1+\dots+s_m)^{\ell+m}} \left(1+\frac{2^j r_B}{(t+s_1+\dots+s_m)^{1/\alpha}}\right)^{-d-\beta}\\
		&\quad \times \left(\frac{r_B}{(t+s_1+\dots+s_m)^{1/\alpha}}\right)^d ds_1\dots ds_m dt \left(\frac{1}{|B|}\int_B |f(x)|^{q_0}dx\right)^{1/q_0}\\
		&\lesssim \int_{r_B^\alpha}^{(2^jr_B)^\alpha} \frac{1}{t^{m+1}} \left(\frac{2^j r_B}{t^{1/\alpha}}\right)^{-d-\beta} \left(\frac{r_B}{t^{1/\alpha}}\right)^d dt\  r_B^{\alpha m} \left(\frac{1}{|B|}\int_B |f(x)|^{q_0}dx\right)^{1/q_0}\\
		&\quad + \int_{(2^jr_B)^\alpha}^\infty \frac{1}{t^{m+1}} \left(\frac{ r_B}{t^{1/\alpha}}\right)^{d}  dt \ r_B^{\alpha m} \left(\frac{1}{|B|}\int_B |f(x)|^{q_0}dx\right)^{1/q_0}\\
		&\lesssim \left(\frac{1}{|B|}\int_B |f(x)|^{q_0}dx\right)^{1/q_0}\left(2^{-j(d+\beta)}r_B^{\alpha m-\beta}\int_{r_B^\alpha}^{(2^jr_B)^\alpha} t^{-(m+1-\beta/\alpha)} dt+r_B^{\alpha m+d}\int_{(2^jr_B)^\alpha}^\infty \frac{dt}{t^{m+1+d/\alpha}}\right)\\
		&\lesssim \left(2^{-j(d+\alpha m)}+2^{-j(d+\beta)}\right)\left(\frac{1}{|B|}\int_B |f(x)|^{q_0}dx\right)^{1/q_0},
	\end{align*}
    provided that $m>\beta/\alpha$. Thus, choosing $m>\beta/\alpha$, both estimates for $F_1$ and $F_2$ are satisfied.
	
	Therefore, for $\gamma= \alpha m \wedge \beta=\beta$, inequality \eqref{eq: bound Sl} holds, and property \ref{itm: aux weighted i} is proved.
\end{proof}

A weight $w$ in $\RR^d$ is said to be in the reverse H\"older class $RH_q(\RR^d)$, with $1<q<\infty$, when there exists $C>0$ such that, for every ball $B\subset\RR^d$, 
\begin{equation*}
  \left( \frac{1}{|B|} \int_B w^q \right)^{1/q}
  \leq \frac{C}{|B|} \int w.
\end{equation*}

From Lemma \ref{lem: aux weighted} and according to \cite[Theorem~6.6]{BZ} (see also \cite[Theorem~2.7]{Na23}), the operator $V_\rho\left(\{t^k\partial_t^k T_t^\alpha\}_{t>0}\right)$ is bounded on $L^p(\RR^d,w)$ provided that $w\in A_{p/p_0}\cap RH_{(q_0/p)'}$ for $1\vee \frac{d}{\beta} <p_0<q_0<\infty$. Actually, this is the same as $w\in A_{p/(1\vee \frac{d}{\beta})}$.

Indeed, set $a_0=1\vee \frac{d}{\beta}$. If $w\in A_{p/p_0}\cap RH_{(q_0/p)'}\subseteq A_{p/p_0}$, since for any $a_0<p_0<p$, $A_{p/p_0}\subseteq A_{p/a_0}$, it is immediate that $w\in A_{p/a_0}$.

Suppose now that $w\in A_{p/a_0}$ for $a_0<p<\infty$. Then, there exists $\delta>0$ such that $w\in A_{p/a_0-\delta}$. We choose $p_0\in (a_0,p)$ verifying $p/p_0>p/a_0-\delta$. Thus, $A_{p/a_0-\delta}\subseteq A_{p/p_0}$ and this yields $w\in A_{p/p_0}$. 

On the other hand, since  $w\in A_{p/a_0}$, we know that there exists $\eta>1$ for which $w\in RH_\eta$. We take $q_0\in (p,\infty)$ such that $1\leq (q_0/p)'<\eta$. Since $RH_\eta \subseteq RH_{(q_0/p)'}$, $w\in RH_{(q_0/p)'}$. 

Therefore, it is proved that $w\in A_{p/p_0}\cap RH_{(q_0/p)'}$ for $1\vee d/\beta <p_0<q_0<\infty$ if and only if $w\in A_{p/(1\vee \frac{d}{\beta})}$.

\section{Proof of Theorem~\ref{thm: necessity rho>2}}

We recall that, for every $\alpha \in (0,2)$, the $\frac{\alpha}{2}$--power of the Laplacian operator can be defined by the Fourier transformation as follows
\[(-\Delta)^{\alpha/2}=\left(|z|^\alpha \hat{f}(z)\right)^{\vee}, \quad f\in S(\RR^d).\]
Thus, $(-\Delta)^{\alpha/2}$ maps $S(\RR^d)$ into $L^2(\RR^d)$.

The operator $-(-\Delta)^{\alpha/2}$ generates in $L^2(\RR^d)$ the semigroup of operators $\{W_t^\alpha\}_{t>0}$ where, for every $t>0$ and $f\in L^2(\RR^d)$,
\[W_t^\alpha(f)=\left(e^{-t|z|^\alpha} \hat{f}(z)\right)^{\vee}.\]

We consider $\phi_\alpha=(e^{-|z|^\alpha})^\vee$. If, for every $t>0$ and $x\in \RR^d$, we set $\phi_{\alpha,t}(x)=\frac{1}{t^\alpha}\phi_\alpha(x/t)$, then $\widehat{\phi_{\alpha,t}}(y)=\widehat{\phi_\alpha}(ty)=e^{-(t|y|)^\alpha}$, for $t>0$ and $y\in \RR^d$. Hence, $\widehat{\phi_{\alpha,t^{1/\alpha}}}(y)=e^{-t|y|^\alpha}$, for every  $t>0$ and $y\in \RR^d$, and we can write
\[W_t^\alpha(f)=\phi_{\alpha,t^{1/\alpha}}*f, \quad f\in L^2(\RR^d), \quad t>0.\]

We now consider the functions $g_i$, $i=1,\dots, d$, as in \cite[p.~32]{CCSj}. For $i=2, \dots, d$ they are given by $g_i(x)=\chi_{[1,2]}(x)$, whereas $g_1$ will be expressed using the Rademacher functions $\{r_k\}_{k=1}^\infty$ supported on $[0,1]$. These functions are defined by
\[r_k=\sum_{j=1}^{2^k-1}\left(\chi_{((2j-2)2^{-k},(2j-1)2^{-k})}-\chi_{((2j+1)2^{-k},2j\ 2^{-k})}\right), \quad k\in \NN.\]
For every $k\in \NN$, we consider the function $q_k$ given by
\[q_k(u)=r_k(u+1)+r_k(u)+r_k(u-1), \,\,\,u\in \mathbb R.\]
Note that, for each $k\in \NN$, $\supp(q_k)\subset [-1,2]$ and $q_k(u)=r_k(u)$ when $u\in (0,1)$.

Let $N\in \NN$ and consider the set $I_N=\{\ell\in \NN: 2N<\ell\leq 3N\}$. The function $g_1$ will be the function
\[g_{1,N}=\sum_{k\in I_N} q_k\]
for $N$ large enough.

By Khinchine's inequality, for every $1<p<\infty$,
\begin{equation}\label{eq: cota norma p}
	\|g_{1,N}\|_{L^p(\RR)}\lesssim \sqrt{N},
\end{equation}
Besides, we easily have that 
\begin{equation}\label{eq: cota puntual g1N}
	|g_{1,N}(u)|\lesssim N, \quad u\in \RR.
\end{equation}

\begin{lem}\label{lem: bound W-1}
	Let $\ell\in \NN$ and $f(x)=\prod_{j=1}^\ell \chi_{[-1,2]}(x_j)$, $x=(x_1,\dots, x_\ell)\in \RR^\ell$. There exists $C>0$ such that
	\begin{equation*}
		|W_t^\alpha(f)(x)-1|\leq Ct, \quad x\in [0,1]^\ell, \quad t>0.
	\end{equation*}
\end{lem}

\begin{proof}Since $\frac{1}{(2\pi)^{\ell/2}}\int_{\RR^\ell} \phi_{\alpha,t}(z)dz=1$, for every $t>0$, we can write
	\[W_t^\alpha(f)(x)-1=\frac{1}{(2\pi)^{\ell/2}}\int_{\RR^\ell} \left(\prod_{j=1}\chi_{[-1,2]}(z_j)-1\right)\phi_{\alpha,t^{1/\alpha}}(x-z)dz, \quad x\in \RR^d.\]
	
	According to \cite[(2.13)]{BDM},
	\[|\phi_{\alpha,t}(z)|\lesssim \frac{1}{t^\ell} \left(\frac{t}{t+|z|}\right)^{\ell+\alpha}, \quad z\in \RR^\ell, t>0.\]
	
	Then, for every $x\in [0,1]^\ell$ and $t>0$, we get
	\begin{align*}
		|W_t^\alpha(f)(x)-1|&\lesssim \int_{\RR^\ell} \left|\prod_{j=1}\chi_{[-1,2]}(z_j)-1\right|\left|\phi_{\alpha,t^{1/\alpha}}(x-z)\right|dz\\
		&\lesssim \int_{\RR^\ell\setminus [-1,2]^\ell} \frac{1}{t^{\ell/\alpha}}\left(\frac{t^{1/\alpha}}{t^{1/\alpha}+|z-x|}\right)^{\ell+\alpha}dz\\
		&\lesssim \frac{1}{t^{\ell/\alpha}} \int_{\RR^\ell\setminus B(x,1)} \left(\frac{t^{1/\alpha}}{t^{1/\alpha}+|z-x|}\right)^{\ell+\alpha}dz\\
		&\lesssim \frac{1}{t^{\ell/\alpha}} \int_1^\infty \left(\frac{t^{1/\alpha}}{t^{1/\alpha}+r}\right)^{\ell+\alpha}r^{\ell-1} dr\\
		&\lesssim \frac{1}{t^{\ell/\alpha}} \int_1^\infty \frac{t^{\ell/\alpha+1}}{r^{\ell+\alpha}} dr\lesssim t, \quad t>0.\qedhere
	\end{align*}
\end{proof}

We now define, for every $\ell\in \NN$ and $g\in L^2(\RR)$,
\[A_\ell^\alpha(g)=W_{2^{-\alpha\ell}}^\alpha(g), \quad D_\ell(g)(x)=2^{\ell-1}\int_{x-2^{-\ell}}^{x+2^{-\ell}} g(y) dy, \quad x\in \RR.\]

As in \cite[p.~21]{CCSj} (see also \cite[\S2]{JW}), we define the relation $\sim_2$ between sequences of operator in $L^2(\RR)$ as follows. Suppose that $\{P_\ell\}_{\ell\in \NN}$ and $\{Q_\ell\}_{\ell\in \NN}$ are two sequences of bounded operators in $L^2(\RR)$. We say that $\{P_\ell\}_{\ell\in \NN}\sim_2\{Q_\ell\}_{\ell\in \NN}$ when, for every sequence of numbers $\{\nu_m\}_{m\in \NN}$ such that $|\nu_m|\leq 1$ for every $m\in \NN$, and every $g\in L^2(\RR)$,
\[\left\|\sum_{\ell\in\NN}\left(P_\ell(g)-Q_\ell(g)\right)\nu_\ell\right\|_{L^2(\RR)}\lesssim \|g\|_{L^2(\RR)}.\]

Suppose that $I\subset (0,\infty)$ and $\{T_t\}_{t\in I}$ is a family operators in $L^p(\RR^d)$ for some $1\le p<\infty$. We define the variation operator $V_{2,I}(\{T_t\}_{t\in I})$ by
\[
V_{2,I}(\{T_t\}_{t\in I})(f)(x)=\sup_{\substack{ t_k<t_{k-1}<\dots<t_1 \\ t_k\in I}}\left(\sum_{j=1}^{k-1}|T_{t_{j+1}}(f)(x)-T_{t_j}(f)(x)|^2\right)^{1/2}.\]
According to \cite[Theorem~2.5]{JW}, if $\{P_\ell\}_{\ell\in \NN}\sim_2\{Q_\ell\}_{\ell\in \NN}$, then, for every $g\in L^2(\RR)$,
\[\left\|\left(\sum_{\ell\in\NN}\left|P_\ell(g)-Q_\ell(g)\right|^2\right)^{1/2}\right\|_{L^2(\RR)}\lesssim \|g\|_{L^2(\RR)},\]
which implies 
\[\left\|V_{2,\NN}\left(\{P_\ell-Q_\ell\}_{\ell\in \NN}\right)(g)\right\|_{L^2(\RR)}\lesssim \|g\|_{L^2(\RR)}.\].

\begin{lem}\label{lem: equivalence}
	Let $\alpha\in (0,2)$. We have that $\{A_\ell^\alpha\}_{\ell\in \NN}\sim_2 \{D_\ell\}_{\ell\in \NN}$.
\end{lem}

\begin{proof}For every $\ell \in \NN$  recall that
	\[\widehat{\phi_{\alpha,2^{-\ell}}}(y)=e^{-\left(2^{-\ell}|y|\right)^\alpha}, \quad y\in \RR.\]
Then, 
\[|1-\widehat{\phi_{\alpha,2^{-\ell}}}(y)|\lesssim \left(2^{-\ell}|y|\right)^\alpha, \quad 2^{-\ell}|y|\leq 1,\]
and
\[|\widehat{\phi_{\alpha,2^{-\ell}}}(y)|\lesssim |2^{-\ell}y|^{-1}, \quad 2^{-\ell}|y|\geq 1.\]

The proof can be finished as the corresponding one in \cite[Lemma~8.4]{CCSj}.
\end{proof}

We can argue as in the proof of \cite[Proposition~8.3]{CCSj} by using Lemma~\ref{lem: equivalence} above, to obtain the following property.

\begin{lem}\label{lem: lim measure 1}
	For some $C>0$, we have that
	\[\left|\left\{x\in (0,1): V_{2,I_N}(\{A_\ell^\alpha\}_{\ell\in \NN})(g_{1,N})(x)>C\sqrt{N\log(\log N)}\right\}\right|\underset{N\to \infty}{\longrightarrow} 1.\]
\end{lem}

We are now going to prove Theorem~\ref{thm: necessity rho>2}. With the definitions given above, let $g_N=g_{1,N}\prod_{i=2}^d g_i$, for $N\in \NN$ large enough. By the subordination formula,
\[W_t^\alpha(g_N)(x)=\int_0^\infty \mu_t^\alpha (s)W_s(g_N)(x)ds, \quad x\in \RR^d.\]

We decompose $W_s(g_N)$, $s>0$, in the following way
\begin{align}\label{eq: W-P}
	\nonumber W_s(g_N)(x)&=W_s(g_{1,N})(x_1)\prod_{i=2}^d W_s(g_i)(x_i)\\
	&\nonumber =W_s(g_{1,N})(x_1)-W_s(g_{1,N})(x_1)\left(1-\prod_{i=2}^d W_s(g_i)(x_i)\right)\\
    &:=W_s(g_{1,N})(x_1)-P_s(g_N)(x), \quad x=(x_1, \dots, x_d)\in [0,1]^d.
\end{align}

By \eqref{eq: cota puntual g1N}, $|W_s(g_{1,N})(x)|\lesssim N$, for $x\in \RR$ and $s>0$.
We deduce that
\begin{align*}
   \left|\int_0^\infty \mu_t^\alpha (s)P_s(g_N)ds\right|
   &\leq \int_0^\infty \mu_t^\alpha (s)|W_s(g_{1,N})(x_1)| \left(1-\prod_{i=2}^d W_s(g_i)(x_i)\right)ds\\
 	&\lesssim N \int_0^\infty \mu_t^\alpha (s)\left(1-\prod_{i=2}^d W_s(g_i)(x_i)\right)ds\\
 	&\sim N \left(1- W_t^\alpha\left(\prod_{i=2}^d g_i\right)(\bar{x})\right), \quad \bar{x}=(x_2, \dots, x_d), \quad (x_1,\bar{x})\in [0,1]^d.
\end{align*}
By using Lemma~\ref{lem: bound W-1} with $\ell=d-1$
\[\left|\int_0^\infty \mu_t^\alpha (s)P_s(g_N)ds\right|\lesssim Nt, \quad x=(x_1,\dots, x_d)\in [0,1]^d, t>0.\]
Hence, for every $t=2^{-\alpha\ell}$ with $\ell\in I_N$, 
\[\left|\int_0^\infty \mu_{2^{-\alpha\ell}}^\alpha (s)P_s(g_N)ds\right|\lesssim N2^{-\alpha N}\lesssim 2^{-\alpha N/2}, \; x=(x_1,\dots, x_d)\in [0,1]^d.\]

Since $I_N$ has $N$ elements, by \cite[(3.3)]{CCSj} we have 
\begin{equation}\label{eq: var decay}
V_{2,I_N}\left(\left\{\int_0^\infty \mu_{2^{-\alpha t}}^\alpha (s)P_s(g_N)ds\right\}_{t\in I_N}\right)\leq C_0 e^{-\alpha N/2}\sqrt{N},
\end{equation}
for $x=(x_1,\dots, x_d)\in [0,1]^d$ and some constant $C_0>0$.

On the other hand, 
\[W_t^\alpha(g_{1,N})(x_1)=\int_0^\infty \mu_t^\alpha(s) W_s(g_{1,N})(x_1) ds, \quad x_1\in \RR,\]
so, for some constant $C_1>0$, by Lemma~\ref{lem: lim measure 1} we know that
\begin{equation}\label{eq: lim measure g1}
		\left|\left\{x_1\in [0,1]: V_{2,I_N}(\{W_{2^{-\alpha\ell}}^\alpha(g_{1,N})\}_{\ell\in I_N})(x_1)>C_1\sqrt{N\log(\log N)}\right\}\right|\underset{N\to \infty}{\longrightarrow} 1.
\end{equation}

From \eqref{eq: W-P} we have
\begin{align*}
    	&\left|\left\{x_1\in [0,1]: V_{2,I_N}(\{W_{2^{-\alpha\ell}}^\alpha(g_{1,N})\}_{\ell\in I_N})(x_1)>C_1\sqrt{N\log(\log N)}\right\}\right| \\
	&\quad = \left|\left\{x\in [0,1]^d: V_{2,I_N}(\{W_{2^{-\alpha\ell}}^\alpha(g_{1,N})\}_{\ell\in I_N})(x_1)>C_1\sqrt{N\log(\log N)}\right\}\right|\\
	&\quad \leq \left|\left\{x\in [0,1]^d: V_{2,I_N}(\{W_{2^{-\alpha\ell}}^\alpha(g_N)\}_{\ell\in I_N})(x)>\frac{C_1}{2}\sqrt{N\log(\log N)}\right\}\right|\\
	&\qquad + \left|\left\{x\in [0,1]^d: V_{2,I_N}\left(\left\{\int_0^\infty \mu_{2^{-\alpha \ell}}^\alpha (s)P_s(g_N)ds\right\}_{\ell\in I_N}\right)>\frac {C_1}{2}\sqrt{N\log(\log N)}\right\}\right|.
\end{align*}
Choosing $N$ large enough such that $C_0e^{-\alpha N/2}\leq \frac {C_1}{2}\sqrt{\log(\log(N))}$, the second term is zero from \eqref{eq: var decay}. Therefore,
\begin{align*}
    &\left|\left\{x_1\in [0,1]: V_{2,I_N}(\{W_{2^{-\alpha\ell}}^\alpha(g_{1,N})\}_{\ell\in I_N})(x_1)>C_1\sqrt{N\log(\log N)}\right\}\right|\\&\leq \left|\left\{x\in [0,1]^d: V_{2,I_N}(\{W_{2^{-\alpha\ell}}^\alpha(g_N)\}_{\ell\in I_N})(x)>\frac{C_1}{2}\sqrt{N\log(\log N)}\right\}\right|\leq 1,
\end{align*}
and from \eqref{eq: lim measure g1}, we deduce that
\begin{equation}\label{eq: lim measure gN}
	\left|\left\{x\in [0,1]^d: V_{2,I_N}(\{W_{2^{-\alpha\ell}}^\alpha(g_N)\}_{\ell\in I_N})(x)>\frac{C_1}{2}\sqrt{N\log(\log N)}\right\}\right|\underset{N\to \infty}{\longrightarrow} 1.
\end{equation}

For any $N\in \NN$, we clearly have
\begin{align*}
	&\left|\left\{x\in [0,1]^d: V_{2,I_N}(\{W_{2^{-\alpha\ell}}^\alpha(g_N)\}_{\ell\in \NN})(x)>\frac{C_1}
    {2}\sqrt{N\log(\log N)}\right\}\right|\\
	&\leq \left|\left\{x\in \RR^d: V_2(\{W_t^\alpha(g_N)\}_{t>0})(x)>\frac{C_1}{2}\sqrt{N\log(\log N)}\right\}\right|.
\end{align*}
 Then, if $1< p<\infty$ and we suppose that $V_2(\{W_t^\alpha\}_{t>0})$ is bounded from  $L^p(\RR^d)$ to $L^{p,\infty}(\RR^d)$, there exists $M>0$ such that
\[\left|\left\{x\in \RR^d: V_2(\{W_t^\alpha(g_N)\}_{t>0})(x)>\frac{C_1}{2}\sqrt{N\log(\log N)}\right\}\right|\leq M \left(\frac{\|g_N\|_{L^p(\RR^d)}}{\sqrt{N\log(\log N)}}\right)^p\]
for every $N\in \NN$.

But, from \eqref{eq: cota norma p},
$\|g_N\|_{L^p(\RR^d)}\lesssim \sqrt{N}$ so
\[\left|\left\{x\in \RR^d: V_2(\{W_t^\alpha(g_N)\}_{t>0})(x)>\frac{C_1}{2}\sqrt{N\log(\log N)}\right\}\right|\underset{N\to \infty}{\longrightarrow} 0\]
which contradicts \eqref{eq: lim measure gN}. Therefore,  $V_2(\{W_t^\alpha\}_{t>0})$ is not bounded from  $L^p(\RR^d)$ to $L^{p,\infty}(\RR^d)$ for any $1< p<\infty$.

Since, as it was proved in Section~\ref{sec: difference}, the operator $V_2(\{T_t^\alpha-W_t^\alpha\}_{t>0})$ is bounded in $L^p(\RR^d)$ for every $1\vee \frac{d}{\beta}<p<\infty$, we conclude that $V_2(\{T_t^\alpha\}_{t>0})$ is not bounded from $L^p(\RR^d)$ into $L^{p,\infty}(\RR^d)$ when $1\vee \frac{d}{\beta}<p<\infty$.

\section*{Acknowledgments}

\subsection*{Funding} The authors are partially supported by grant PID2023-148028NB-I00 from the Spanish Government. The second author is also partially supported by grants PICT-2019-2019-00389 (Agencia Nacional de Promoción Científica y Tecnológica), PIP-1220200101916O (Consejo Nacional de Investigaciones Científicas y Técnicas) and CAI+D 2019-015 (Universidad Nacional del Litoral).

\bibliographystyle{acm}

\end{document}